\documentclass{compositio}
\usepackage{graphicx}
\usepackage{color}
\usepackage{amscd}
\usepackage{latexsym} 
\usepackage{mathrsfs} 
\usepackage{epsfig}
 \usepackage{amsmath}
\usepackage[all]{xy}
 \usepackage{float}
 
\newcommand\PP{\mathbb P}
\newcommand\C{\mathbb C}
\newcommand\Q{\mathbb Q}
\newcommand\R{\mathbb R}
\newcommand\Z{\mathbb Z}
\newcommand\N{\mathbb N}

\newcommand{\V}{\mathbb V}
\newcommand{\U}{\mathbb U}

\newcommand{\eps}{\varepsilon}

\renewcommand{\d}{\mathfrak{d}}
\newcommand{\D}{\mathfrak{D}}
\newcommand{\m}{\mathfrak{m}}

\newcommand\Aut{\operatorname{Aut}}

\newcommand\Ad{\operatorname{Ad}}
\newcommand\ad{\operatorname{ad}}

\newcommand{\bra}{\langle}
\newcommand{\ket}{\rangle}
\newcommand{\wed}{\wedge}
\newcommand{\del}{\partial}

\newcommand{\li}{\operatorname{Li}}
\renewcommand{\S}{\operatorname{S}}

\renewcommand{\th}{\theta}

\newcommand{\g}{\mathfrak{g}}

\makeatletter \@addtoreset{equation}{section} \makeatother
\renewcommand{\theequation}{\thesection.\arabic{equation}}
\newtheorem{thm}{Theorem}[section]
\newtheorem{prop}[thm]{Proposition}
\newtheorem{lem}[thm]{Lemma}
\newtheorem{cor}[thm]{Corollary}

\newtheorem{defn}[thm]{Definition}
\newenvironment{exa}{\noindent\textbf{Example}.}{\\}
\newenvironment{rmk}{\noindent\textbf{Remark}.}{\\}

\begin{document}

\title[]{Block-G\"ottsche invariants from wall-crossing}
\author{S. A. Filippini}
\email{saraangela.filippini@math.uzh.ch}
\address{Institut f\"ur Mathematik\\
Universit\"at Z\"urich, Wintherthurerstrasse 190, 8057 Z\"urich, Schweiz.}

\author{J. Stoppa}
\email{jacopo.stoppa@unipv.it}
\address{Dipartimento di Matematica ``F. Casorati"\\
Universit\`a di Pavia, via Ferrata 1, 27100 Pavia, Italia.
}
\date{}

\dedication{Dedicated to the memory of Kentaro Nagao}

\thanks{This research was supported in part by the Hausdorff Institute for Mathematics, Bonn (JHRT ``Mathematical Physics") and GRIFGA 2012-2015. The research leading to these results has received funding from the European Research Council under the European Union's Seventh Framework Programme (FP7/2007-2013) / ERC Grant agreement no. 307119.} 

\begin{abstract} We show how some of the refined tropical counts of Block and G\"ottsche emerge from the wall-crossing formalism. This leads naturally to a definition of a class of putative $q$-deformed Gromov-Witten invariants. We prove that this coincides with another natural $q$-deformation, provided by a result of Reineke and Weist in the context of quiver representations, when the latter is well defined. 
\end{abstract}

\maketitle

\section{Introduction} 

Recently Block and G\"ottsche \cite{gottsche} (see also the earlier accounts in \cite{vivek} section 6 and \cite{mik} section 1) introduced a refined tropical count for plane tropical curves, where the usual Mikhalkin multiplicity is replaced by a function taking values in Laurent polynomials in one variable. The original motivation for Block and G\"ottsche's proposal is connected with a generalization of G\"ottsche's conjecture and a refinement of Severi degrees (see \cite{gott} for the original conjecture, and \cite{vivek} for a discussion of its proof(s) and conjectural refinements). The tropical invariance of such counts was proved by Itenberg and Mikhalkin \cite{mik}. This invariance is perhaps surprising from a purely tropical point of view. 

The first purpose of this paper is to point out a different perspective from which the definition of the Block-G\"ottsche multiplicity and the invariance of some of the associated refined tropical counts look completely natural. This point of view is provided by the wall-crossing formula for refined Donaldson-Thomas invariants. 

In section \ref{background} we recall a method due to Gross, Pandharipande and Siebert \cite{gps} (based on the tropical vertex group $\V$) that allows to express the wall-crossing of numerical Donaldson-Thomas invariants\footnote{This method is very general and also works for factorizations in $\V$ which do not correspond to some wall-crossing. We also point out for experts that we are referring here to a \emph{wall-crossing of the first kind} in the sense of \cite{ks} section 2.3 (in the language of mathematical physics, this corresponds to crossing the wall of marginal stability). Although we will not discuss this in the present paper, the group $\mathbb{V}$ and its $q$-deformation introduced in section 3 below also play an important role in \emph{wall-crossing of the second kind} (\cite{ks} section 2.4), corresponding to a change of $t$-structure in a triangulated category, e.g. by tilting. See \cite{kentaro} for a precise result in the context of quiver mutations.} in terms of invariants enumerating \emph{rational} tropical curves; the fundamental example is given by formulae \eqref{IntroCommutator} and \eqref{VtoTrop} below. At the end of section \ref{background} we motivate the need to go from numerical to refined wall-crossing formulae: this is natural from the Donaldson-Thomas point of view (the main reference being the work of Kontsevich and Soibelman \cite{ks}). The details of the $q$-deformation are given in section \ref{quantum}. In section \ref{scattering} we show that the definition of the Block-G\"ottsche multiplicity, and the invariance of the related counts for rational tropical curves, are essentially equivalent to an extension of the methods of \cite{gps} sections 1 and 2 to the refined wall-crossing formulae. This is summarized in Corollary \ref{ScatterToTrop}.  In section \ref{mpsSection} (Proposition \ref{qVtoTropSpec}) we finally give the refinement of the basic GPS formula \eqref{VtoTrop}. For the sake of completeness in the Appendix we sketch the proof of the invariance of $q$-deformed tropical counts for rational curves.  

The other main theme of this paper is that the approach above leads naturally to a $q$-deformation of a class of Gromov-Witten invariants. Indeed \cite{gps} shows that computing commutators in $\V$ is \emph{equivalent} to the calculation of a class of genus zero Gromov-Witten invariants $N[({\bf P}_1, {\bf P}_2)]$ of (blowups of) weighted projective planes, parametrized by a pair of partitions $({\bf P}_1, {\bf P}_2)$ (see \eqref{gpsThm} and \eqref{gpsThm2} below). By adapting this argument to the $q$-deformed case we find a natural $q$-deformation of $N[({\bf P}_1, {\bf P}_2)]$ in terms of Block-G\"ottsche counts, namely Definition \ref{refinedGW}. On the other hand when the vector $(|{\bf P}_1|, |{\bf P}_2|)$ is primitive $N[({\bf P}_1, {\bf P}_2)]$ is a BPS invariant in the sense of \cite{gps} section 6.3. In this special case a result of Reineke and Weist \cite{reinweist} shows that $N[({\bf P}_1, {\bf P}_2)]$ equals the Euler characteristic of a moduli space of quiver representations $\mathcal{M}({\bf P}_1, {\bf P}_2)$. A different choice of $q$-deformation is then the symmetrized Poincar\'e polynomial $q^{-\frac{1}{2}\dim \mathcal{M}({\bf P}_1, {\bf P}_2)}P(\mathcal{M}({\bf P}_1, {\bf P}_2))(q)$. Our main result in this connection is Theorem \ref{comparison} which shows that the two choices coincide. A key ingredient is the Manschot-Pioline-Sen formula \cite{mps}; indeed it follows from the proof of Theorem \ref{comparison} that the MPS formula in this context can be interpreted precisely as the equality of the two ``quantizations".  We conclude section \ref{mpsSection} with some further remarks on the $q$-deformed Gromov-Witten invariants, touching on explicit formulae, integrality, and the connection with refined Severi degrees and motivic Donaldson-Thomas invariants.
\acknowledgements{We are grateful to Pierrick Bousseau, Lothar G\"ottsche, Rahul Pandharipande, Mar\-kus Reineke, Richard Thomas, Thorsten Weist for some very helpful discussions, as well as to the anonymous Referee for suggesting many improvements.}

\section{Tropical vertex}\label{background}

The tropical vertex group $\mathbb{V}$ is a subgroup of the group of formal $1$-parameter families of automorphisms of the complex algebraic torus $\C^*\times\C^*$, $\mathbb{V} \subset \Aut_{\C[[t]]}\C[x, x^{-1}, y, y^{-1}][[t]]$. Fix integers $a, b$ and a function $f_{(a, b)} \in \C[x, x^{-1}, y, y^{-1}][[t]]$ of the form
\begin{equation}\label{admissibleFunct}
f_{(a, b)} = 1 + t\,x^a y^b\,g(x^a y^b, t)
\end{equation}
for a formal power series $g \in \C[z][[t]]$. To this we attach an element $\theta_{(a, b), f_{(a, b)}} \in \mathbb{V}$ defined by
\begin{equation*}
\theta_{(a, b), f_{(a, b)}}(x) = x\,f^{-b}_{(a, b)}, \quad \theta_{(a, b), f_{(a, b)}}(y) = y\,f^a_{(a, b)}. 
\end{equation*}

Then we can define $\mathbb{V}$ as the completion with respect to $(t)\subset\C[[t]]$ of the subgroup of $\Aut_{\C[[t]]}\C[x, x^{-1}, y, y^{-1}][[t]]$ generated by all the transformations $\theta_{(a, b), f_{(a, b)}}$ (as $(a, b)$ and $f_{(a, b)}$ vary). Elements of $\mathbb{V}$ are ``formal symplectomorphisms", i.e. they preserve the holomorphic symplectic form $\frac{dx}{x}\wed\frac{dy}{y}$. The basic question about $\mathbb{V}$ is to compute a general commutator, 
$$[\theta_{(a_1, b_1), f_1}, \theta_{(a_2, b_2), f_2}] = \theta^{-1}_{(a_2, b_2), f_2} \theta_{(a_1, b_1), f_1} \theta_{(a_2, b_2), f_2} \theta^{-1}_{(a_1, b_1), f_1}.$$
Despite its elementary flavour, it turns out that this problem plays a crucial role in a number of contexts in algebraic geometry, most importantly for us in wall-crossing formulae for counting invariants in abelian and triangulated categories (see \cite{ks}). Suppose for definiteness that $a_1, b_1, a_2, b_2$ are all nonnegative, and that $(a_1, b_1)$ follows $(a_2, b_2)$ in clockwise order. Then there exists a unique, possibly infinite (but countable) collection of \emph{primitive} vectors $(a, b)$ with positive entries, and attached functions $f_{(a, b)}$ (of the form \eqref{admissibleFunct}) such that
$\theta^{-1}_{(a_2, b_2), f_2} \theta_{(a_1, b_1), f_1} \theta_{(a_2, b_2), f_2} \theta^{-1}_{(a_1, b_1), f_1} = \prod^{\to}_{(a, b)}\theta_{(a, b), f_{(a, b)}}$.
Here $\prod^{\to}$ denotes the operation of writing products of finite subcollections of group elements $\theta_{(a, b), f_{(a, b)}}$ from left to right so that the slopes of $(a, b)$ in $\R^2$ are decreasing (i.e. in clockwise order), and then taking the direct limit over all finite collections. Gross, Pandharipande and Siebert have shown that the problem of computing the functions $f_{(a, b)}$ carries a surprisingly rich intrinsic geometry, which involves the virtual counts of rational curves in weighted projective planes with prescribed singularities and tangencies. To formulate the simplest result of this type, we fix two integers $\ell_1, \ell_2$ and consider the transformations $\theta_{(1, 0), (1+ tx)^{\ell_1}}, \theta_{(0, 1), (1+ty)^{\ell_2}}$. Let us define functions $f_{(a, b)}$ as above (in particular, for $(a, b)$ primitive) by 
\begin{equation}\label{IntroCommutator}
[\theta_{(1, 0), (1+ tx)^{\ell_1}}, \theta_{(0, 1), (1+ ty)^{\ell_2}}] = \prod^{\to}_{(a, b)}\theta_{(a, b), f_{(a, b)}}.
\end{equation}
Since $f_{(a, b)}$ has the form \eqref{admissibleFunct} we may take its logarithm as a formal power series, which must then be of the form $\log f_{(a, b)} = \sum_{k \geq 0} c^{(a, b)}_k (tx)^{ak}(ty)^{bk}$. Let us write ${\bf P}$ for an ordered partition and $|{\bf P}|$ for its size (the sum of all its parts). Theorem 0.1 of \cite{gps} gives a formula for the coefficients $c^{(a, b)}_k$ in terms of certain Gromov-Witten invariants $N_{(a, b)}{[({\bf P}_a, {\bf P}_b)]} \in \Q$,
\begin{equation}\label{gpsThm}
c^{(a, b)}_k = k \sum_{|{\bf P}_a| = k a} \sum_{|{\bf P}_b| = kb} N_{(a, b)}{[({\bf P}_a, {\bf P}_b)]},
\end{equation}
where the length of ${\bf P}_a$ (respectively ${\bf P}_b$) is $\ell_1$ (respectively $\ell_2$). Here $N_{(a, b)}{[({\bf P}_a, {\bf P}_b)]}$ is the virtual count of rational curves contained in the weighted projective plane $\PP(a, b, 1)$, which must have prescribed singular points along the two toric divisors $D_1, D_2$ dual to the rays spanned by $(-1, 0)$ and $(0, -1)$ respectively, lying away from the torus fixed points, and with multiplicities specified by the ordered partitions ${\bf P}_a, {\bf P}_b$. To make this rigorous one blows up a number of fixed points on $D_1, D_2$ and imposes a suitable degree condition. Moreover one has to make sense of Gromov-Witten theory away from the torus fixed points. See \cite{gps} Section 0.4 for a precise definition. Additionally the curves must be tangent to order $k$ (at an unspecified point) to the divisor $D_{\rm out}$ dual to the ray spanned by $(a, b)$.

The equality \eqref{gpsThm} actually arises from the enumeration of certain plane tropical curves. Consider a \emph{weight vector} ${\bf w} = ({\bf w}_1, {\bf w}_2)$, where each ${\bf w}_i$ is a collection of integers $w_{ij}$ (for $1 \leq i \leq 2$ and $1 \leq j \leq l_i$) such that  $1 \leq w_{i1} \leq w_{i2} \leq \dots \leq w_{i{l_i}}$. For  $1 \leq j \leq l_1$ choose a general collection of parallel lines $\mathfrak{d}_{1j}$ in the direction $(1, 0)$, respectively $\mathfrak{d}_{2j}$ in the direction $(0, 1)$ for $1 \leq j \leq l_2$. We attach the weight $w_{ij}$ to the line $\mathfrak{d}_{ij}$, and think of the lines $\mathfrak{d}_{ij}$ as ``incoming" unbounded edges for connected, rational tropical curves $\Upsilon \subset \R^2$. We prescribe that such curves $\Upsilon$ have a single additional ``outgoing" unbounded edge in the direction $(|{\bf w}_1|,|{\bf w}_2|)$. Let us denote by $\mathcal{S}({\bf w})$ the finite set of such tropical curves $\Upsilon$ (for a general, fixed choice of ends $\mathfrak{d}_{ij}$). Let $(|{\bf w}_1|, |{\bf w}_2|) = (k a, k b)$ for some positive integer $k$ and primitive $(a, b)$. We denote by $N^{\rm trop}_{(a, b)}({\bf w}) = \#^{\mu} \mathcal{S}({\bf w})$ the tropical count of curves $\Upsilon$ as above, i.e. the number of elements of $\mathcal{S}({\bf w})$ counted with the usual multiplicity $\mu$ of tropical geometry (see \cite{mikhalkin}). It is known that $\#^{\mu} \mathcal{S}({\bf w})$ does not depend on the general choice of unbounded edges $\mathfrak{d}_{ij}$ (see \cite{mikhalkin}, \cite{gathmark}). An application of \cite{gps} Theorem 2.8 gives 
\begin{equation}\label{VtoTrop}
c^{(a, b)}_k = k \sum_{|{\bf P}_a| = k a} \sum_{|{\bf P}_b| = kb} \sum_{{\bf w}} \prod^2_{i = 1} \frac{R_{{\bf P}_i | {\bf w}_i}}{|\Aut({\bf w}_i)|} N^{\rm trop}_{(a, b)}({\bf w}),
\end{equation}
where the inner sum is over weight vectors ${\bf w}$ such that $|{\bf w}_i | = {\bf P}_i$ and $R_{{\bf P}_i | {\bf w}_i}$, $|\Aut({\bf w}_i)|$ are certain combinatorial coefficients. The connection to Gromov-Witten theory is established through the identity
\begin{equation}\label{TropToGW}
N_{(a, b)}{[({\bf P}_a, {\bf P}_b)]} = \sum_{{\bf w}} \prod^2_{i = 1} \frac{R_{{\bf P}_i | {\bf w}_i}}{|\Aut({\bf w}_i)|} N^{\rm trop}_{(a, b)}({\bf w}),
\end{equation}
which follows from \cite{gps} Theorems 3.4, 4.4 and Proposition 5.3. In fact, the invariants $N_{(a, b)}{[({\bf P}_a, {\bf P}_b)]}$ are completely determined by factorizations in an extended tropical vertex group. Introducing auxiliary variables $s_1, \ldots, s_{\ell_1}$, $t_1, \ldots, t_{\ell_2}$ one considers the problem of computing the commutator $[\prod^{\ell_1}_{i=1} \theta_{(1, 0), 1 + s_i x}, \prod^{\ell_2}_{j=1} \theta_{(0, 1), 1 + t_j y}]$, with the obvious extension of the notation introduced above. Then as shown in Theorem 5.4 of \cite{gps} one can refine \eqref{gpsThm} to show that the the corresponding weight functions satisfy
\begin{equation}\label{gpsThm2}
\log f_{(a, b)} = k \sum_{|{\bf P}_a| = k a} \sum_{|{\bf P}_b| = kb} N_{(a, b)}{[({\bf P}_a, {\bf P}_b)]} s^{{\bf P}_a} t^{{\bf P}_b} x^{ka}y^{kb}.
\end{equation}
Now we make the basic observation that operators such as $\theta_1 = \theta_{(1, 0), (1+ tx)^{\ell_1}}$ and $\theta_2 = \theta_{(0, 1), (1+ ty)^{\ell_2}}$ admit natural $q$-deformations or ``quantizations", acting on the $\C[[t]]$-algebra generated by quantum variables $\hat{x} \hat{y} = q \hat{y}\hat{x}$. This is motivated by their special role in Donaldson-Thomas theory, where they represent the action of a stable spherical object (\cite{ks} section 6.4). Roughly speaking then $q$-deforming corresponds to passing from Euler characteristics of moduli spaces to their Poincar\'e polynomials. More generally in mathematical physics these operators reflect the spetrum of BPS states of theories belonging to a suitable class, and their $q$-deformation is connected with refined indices counting such states (see e.g. \cite{pioline} for an introduction to this circle of ideas). In the present simple example $\hat{\theta}_1(\hat{x}) = \hat{x}, \hat{\theta}_1(\hat{y}) = \hat{y}(1 + q^{\frac{1}{2}} t \hat{x})^{\ell_1}$, $\hat{\theta}_2(\hat{x}) = \hat{x}(1 + q^{\frac{1}{2}} t\hat{y})^{-\ell_2}, \hat{\theta}_2(\hat{y}) = \hat{y}$. It is then natural to guess the existence of a $q$-deformation of the factorization $\eqref{IntroCommutator}$ for $[\hat{\theta}_1, \hat{\theta}_2]$, as well as of a $q$-analogue of \eqref{VtoTrop}. From the form of \eqref{VtoTrop} one may envisage the existence of $q$-deformed tropical counts $\widehat{N}^{\rm trop}_{(a, b)}({\bf w})$, which should be defined as $ \#^{\mu_q} \mathcal{S}({\bf w})$ for some $q$-deformation of the usual tropical multiplicity. We will see in section \ref{scattering} that this is precisely what happens: $\mu_q$ turns out to be the Block-G\"ottsche multiplicity. Another advantage of this point of view is that \eqref{TropToGW} immediately suggests the form of some putative $q$-deformed Gromov-Witten invariants. We will discuss this in section \ref{mpsSection}.
 
\section{$q$-deformation}\label{quantum}

We will need a more general incarnation of the group $\mathbb{V}$. Let $R$ be a commutative $\C$-algebra which is either complete local or Artinian. Let $\Gamma$ be a fixed lattice with an antisymmetric, bilinear form $\bra -, - \ket$. Consider the infinite dimensional complex Lie algebra $\g$ generated by $e_{\alpha}, \alpha \in \Gamma$, with bracket 
\begin{equation}\label{classBracket}
[e_{\alpha}, e_{\beta}] = \bra \alpha, \beta\ket e_{\alpha + \beta}.
\end{equation}
We also endow $\g$ with the associative, commutative product determined by 
\begin{equation}\label{classProduct}
e_{\alpha} e_{\beta} = e_{\alpha + \beta}.
\end{equation}
With the product \eqref{classProduct} and the bracket \eqref{classBracket} $\g$ becomes a Poisson algebra: the linear map $[x, -]$ satisfies the Leibniz rule (this is a straighforward check on the generators). We write $\widehat{\g}$ for the completed tensor product of $\g$ with $R$, $\widehat{\g} = \g\,\widehat{\otimes}_{\C} R = \lim_{\to} \g \otimes_{\C} R/{\mathfrak{m}^k_R}$, and extend the Poisson structure to $\widehat{\g}$ by $R$-linearity. We denote by $\mathfrak{m}_R[e_{\alpha}]$ the closure of the subalgebra generated by $\mathfrak{m}_R$ and $e_{\alpha}$, and we write $\mathfrak{m}_R[e_{\alpha}] e_{\alpha}$ for the subset of elements of the form $\xi e_{\alpha}$ with $\xi \in \mathfrak{m}_R[e_{\alpha}]$. Let $f_{\alpha} \in \widehat{\g}$ be an element of the form
\begin{equation}\label{fFunctions}
f_{\alpha} \in 1 +  \mathfrak{m}_R[e_{\alpha}] e_{\alpha},
\end{equation}
Let us introduce a class of Poisson automorphisms $\theta_{\alpha, f_{\alpha}}$ of the $R$-algebra $\widehat{\g}$ by prescribing
\begin{equation}\label{classActionEq}
\theta_{\alpha, f_{\alpha}}(e_{\beta}) = e_{\beta} f_{\alpha}^{\bra \alpha, \beta \ket}.
\end{equation}   
Notice that the inverse automorphism $\theta^{-1}_{\alpha, f_{\alpha}}$ is given by $\theta_{\alpha, f^{-1}_{\alpha}}$. More generally for $\Omega \in \Q$ we denote by $\theta^{\Omega}_{\alpha, f_{\alpha}}$ the automorphism $\theta_{\alpha, f^{\Omega}_{\alpha}}$. 
\begin{defn} The \emph{tropical vertex group} $\mathbb{V}_{\Gamma, R}$ is the completion with respect to $\mathfrak{m}_R\subset R$ of the subgroup of $\Aut_{R}(\widehat{\g})$ generated by all the transformations $\theta_{\alpha, f_{\alpha}}$ (as $\alpha$ varies in $\Gamma$ and $f_{\alpha}$ among functions of the form \eqref{fFunctions}).
\end{defn}
\begin{lem}\label{iso} Suppose that $\Gamma$ is the lattice $\Z^2$ endowed with its standard antisymmetric form $\bra (p, q), (p', q')\ket = p q' - q p'$. Then $\mathbb{V}_{\Gamma, \C[[t]]} \cong \mathbb{V}$. 
\end{lem}
\begin{proof} We can identify $\C[x, x^{-1}, y, y^{-1}]$ with $\g$ by the isomorphism $\iota$ defined by $\iota(x) = e_{(1, 0)}$, $\iota(y) = e_{(0, 1)}$. Following the general notation introduced in \eqref{fFunctions}, we let $\mathfrak{m}_{\C[[t]]}[x^a y^b]$ denote the set of formal power series in $t$, divisible by $t$, whose coefficients are polynomials in $x^a y^b$. Similarly $x^a y^b \mathfrak{m}_{\C[[t]]}[x^a y^b]$ denotes the subset of $\mathfrak{m}_{\C[[t]]}[x^a y^b]$ given by formal power series whose coefficients are divisible by $x^a y^b$.  Taking $\widehat{\otimes} \C[[t]]$ on both sides, $f_{(a, b)} \in 1 + x^a y^b \mathfrak{m}_{\C[[t]]}[x^a y^b]$ is mapped to some $f_{\alpha} \in 1 +  \mathfrak{m}_{\C[[t]]}[e_{\alpha}] e_{\alpha}$ where $\alpha = (a, b)$, and by \eqref{classActionEq} we have $\iota^{-1} \circ \theta_{\alpha, f_{\alpha}} \circ \iota = \theta_{(a, b), f_{a, b}}$. This proves the claim since $\theta_{(a, b), f_{(a, b)}}$, $\theta_{\alpha, f_{\alpha}}$ are topological generators.
\end{proof}
\noindent Elements of $\V_{\Gamma, R}$ of the form $\theta_{\alpha, 1 + \sigma e_{\alpha}}$ with $\sigma \in \m_R$ play a special role, as they have a natural interpretation in Donaldson-Thomas theory and in the mathematical physics of BPS states. For example, in the latter context, it has been argued in \cite{gmn} that the quantum-corrected Lagrangian of a class of supersymmetric gauge theories in four dimensions can be determined exactly (in the low-energy limit) as a function of the spectrum of BPS particles through a suitable Riemann-Hilbert factorization problem, whose jump factors are products of the group elements $\theta_{\alpha, 1 + \sigma e_{\alpha}}$ (one for each BPS particle of charge $\alpha$, appearing as many times as its multiplicity or ``index"). In particular as we will see the group elements $\theta_{\alpha, 1 + \sigma e_{\alpha}}$ have a well-defined $q$-deformation (which is also predicted by the physics of BPS states, in terms of ``refined indices" for BPS particles; see e.g. \cite{tudor} for more on this perspective). Accordingly we can give a definition of the subgroup of $\V_{\Gamma, R}$ which is relevant to the wall-crossing of Donaldson-Thomas invariants.
\begin{defn} The \emph{wall-crossing group} $\widetilde{\V}_{\Gamma, R} \subset \V_{\Gamma, R}$ is the completion of the subgroup generated by automorphisms of the form $\theta^{\Omega}_{\alpha, 1 + \sigma e_{\alpha}}$ for $\alpha \in \Gamma$, $\sigma \in \m_R$ and $\Omega\in\Q$. 
\end{defn} 
\noindent An argument in \cite{gps} Section 1 implies that we do not lose too much by restricting to groups of the form $\widetilde{\V}_{\Gamma, S}$, provided we work over a suitably large ring $S$. The argument bears on the case when $R = \C[[t_1, \ldots, t_n]]$ and $f_{\alpha}$ is of the form $f_{\alpha} = 1 + t_i e_{\alpha} g(e_{\alpha}, t_i)$ for some $i = 1, \ldots, n$ and $g \in \C[z][[t_i]]$. First we work modulo $\mathfrak{m}^{k+1}_R$ for some $k$: this will be the order of approximation. Then over $R_k = R/(t^{k+1}_1, \ldots, t^{k+1}_n)$ we have
\begin{equation}\label{logf}
\log f_{\alpha} = \sum^k_{j = 1} \sum_{w\geq 1} w a_{i j w} e_{w\alpha} t^k_i 
\end{equation}
for some coefficients $a_{i j w}$ which vanish for all but finitely many $w$. We introduce variables $u_{i j}$ for $i = 1, \ldots, n$, $j = 1, \ldots, k$, and pass to the base ring $\widetilde{R}_k = \C[u_{i j}]/(u^2_{ij})$. There is an inclusion $i\!: R_k \hookrightarrow \widetilde{R}_k$ given by $t_i \mapsto \sum^k_{j = 1} u_{i j}$. In particular we have an inclusion of groups $\mathbb{V}_{\Gamma, R_k} \hookrightarrow \mathbb{V}_{\Gamma, \widetilde{R}_k}$. We will often be sloppy and identify an element of $R_k$ with its image under $i$. As $u^2_{ij} = 0$, the image of \eqref{logf} under $i$ is 
\begin{equation*}
\log f_{\alpha} = \sum^k_{j = 1}\sum_{\#J = j} \sum_{w\geq 1} j! w a_{i j w} \prod_{l\in J} u_{i l} e_{w\alpha}, 
\end{equation*}
summing over $J \subset \{1, \ldots, n\}$. Exponentiating both sides and using again $u^2_{ij} = 0$ we find the factorization 
\begin{equation*}
f_{\alpha} = \prod^k_{j = 1}\prod_{\#J = j} \prod_{w\geq 1} \left(1 + j! w a_{i j w} \prod_{l\in J} u_{i l} e_{w\alpha}\right).
\end{equation*}
So for $f_{\alpha}$ as above we have found a factorization in $\mathbb{V}_{\Gamma, \widetilde{R}_k}$  
\begin{equation}\label{factorization}
\theta_{\alpha, f_{\alpha}} \equiv \prod^k_{j = 1}\prod_{\#J = j} \prod_{w\geq 1} \theta_{\alpha, f_{j J w}} \operatorname{mod} t^{k+1}_i,\quad f_{j J w} = 1 + j! w a_{i j w} \prod_{l\in J} u_{i l} e_{w\alpha}.
\end{equation}
Notice that we have 
\begin{equation*}
\theta_{\alpha, f_{j J w}} = \theta_{w\alpha, 1 + j! a_{i j w} \prod_{l\in J} u_{i l} e_{w\alpha}}.
\end{equation*}
 
Following \cite{ks}, we can take advantage of the Poisson structure on $\widehat{\g}$ to give a different expression for the special transformations $\theta_{\alpha, 1 + \sigma e_{m\alpha}}$, which leads easily to their $q$-deformation. Fix $\sigma \in \mathfrak{m}_R$, and define the \emph{dilogarithm} $\li_2(\sigma e_{\alpha})$ by
\begin{equation*}
\li_2(\sigma e_{\alpha}) = \sum_{k \geq 1} \frac{\sigma^k e_{k\alpha}}{k^2}.
\end{equation*} 
This is well defined by our assumptions on $R$. Then $\ad(\li_2(\sigma e_{\alpha})) = [\li_2(\sigma e_{\alpha}), - ]$ is a derivation of $\widehat{\g}$, and again by our assumptions on $R$ its exponential is a well defined Poisson automorphism of $\widehat{\g}$, acting by
\begin{equation*}
\exp(\ad(\li_2(\sigma e_{\alpha})))(e_{\beta}) = \sum_{h\geq 0} \frac{1}{h!}\ad^h(\li_2(\sigma e_{\alpha}))(e_{\beta}). 
\end{equation*}
\begin{lem}\label{classAction} The automorphism $\theta_{\alpha, 1 + \sigma e_{m\alpha}}$ equals $\exp(\frac{1}{m}\ad(\li_2(-\sigma e_{m\alpha})))$. 
\end{lem}
\begin{proof} We have
\begin{align*}
[\frac{1}{m}\li_2(-\sigma e_{m\alpha}), e_{\beta}] &= \frac{1}{m}\sum_{k \geq 1} (-1)^k\frac{\sigma^k}{k^2} [e_{k m \alpha}, e_{\beta}] = \frac{1}{m}\sum_{k \geq 1} (-1)^k\frac{\sigma^k}{k^2} \bra k m\alpha, \beta\ket e_{k m \alpha + \beta}\\
&= \sum_{k \geq 1} (-1)^k\frac{\sigma^k}{k} \bra\alpha, \beta\ket e_{\beta} e_{k m \alpha} = e_{\beta} \bra\alpha, \beta\ket \sum_{k \geq 1} (-1)^k\frac{\sigma^k e_{k m \alpha}}{k}\\
&= e_{\beta} \bra\alpha, \beta\ket \log(1 + \sigma e_{m \alpha}).
\end{align*}
Using the Leibniz rule and induction, we find
\begin{equation*}
\ad^h(\frac{1}{m}\li_2(-\sigma e_{m\alpha}))(e_{\beta}) = e_{\beta} (\bra\alpha, \beta\ket \log(1 + \sigma e_{m\alpha}))^h,
\end{equation*}
and the result follows.
\end{proof}
\noindent We now replace $\g$ with an associative, noncommutative algebra $\g_q$ over the ring 
\begin{equation*}
\C[q^{\pm \frac{1}{2}}, ((q^n - 1)^{-1})_{n \geq 1}],
\end{equation*} 
generated by symbols $\hat{e}_{\alpha}, \alpha \in \Gamma$. The classical product \eqref{classProduct} is quantized to
\begin{equation}\label{qProduct}
\hat{e}_{\alpha} \hat{e}_{\beta} = q^{\frac{1}{2}\bra\alpha, \beta\ket}\hat{e}_{\alpha + \beta}.
\end{equation}
As standard in the quantization the Lie bracket is the natural one given by the commutator. In other words we are now thinking of the $\hat{e}_{\alpha}$ as operators (as opposed to the classical bracket \eqref{classBracket}, which corresponds to a Poisson bracket of the $e_{\alpha}$ seen as functions). Namely, we set
\begin{equation}\label{qBracket}
[\hat{e}_{\alpha}, \hat{e}_{\beta}] := (q^{\frac{1}{2}\bra\alpha, \beta\ket} - q^{-\frac{1}{2}\bra\alpha, \beta\ket})\hat{e}_{\alpha + \beta}. 
\end{equation}
Since this is the commutator bracket of an associative algebra, $\g_q$ is automatically Poisson. 

Suppose that an element $\xi \in \g$ lies in the $\C$-subalgebra generated by $q^{\pm \frac{1}{2}}$ and $\hat{e}_{\alpha}, \alpha \in \Gamma$. Then we say that $\xi$ \emph{admits a quasi-classical limit}\footnote{This is the opposite of the quasi-classical limit considered in \cite{ks}. It gives the wrong sign for BPS states counts, but it is the natural choice for the purposes of this paper.}, which is obtained when we specialize $q^{\frac{1}{2}}$ to $1$. As an example, notice that after rescaling the Lie bracket \eqref{qBracket} admits the quasi-classical limit \eqref{classBracket}:
\begin{equation*}
\lim_{q^{\frac{1}{2}}\to 1} \frac{1}{q-1} [\hat{e}_{\alpha}, \hat{e}_{\beta}] = \bra \alpha, \beta\ket \hat{e}_{\alpha + \beta}.
\end{equation*}
\noindent Fixing a local complete or Artinian $\C$-algebra $R$ with maximal ideal $\mathfrak{m}_R$ as usual, we define $\widehat{\g}_q = \g_q\,\widehat{\otimes}_{\C} R$. The fundamental case is $\g_q[[t]]$ where $t$ is a central variable. Again, if $\xi \in \widehat{\g}_q$ lies in the closure of the $R$-subalgebra generated by $q^{\pm \frac{1}{2}}$ and $\hat{e}_{\alpha}, \alpha \in \Gamma$ (e.g. when $R = \C[[t]]$, $\xi$ is a formal power series in $t$ whose coefficients are Laurent polynomials in $q^{\frac{1}{2}}$ and the $\hat{e}_{\alpha}$), it admits a quasi-classical limit given by specializing $q^{\frac{1}{2}}$ to $1$.

By Lemma \ref{classAction} the action of $\theta_{\alpha, 1 + \sigma e_{\alpha}}$ on $\widehat{\g}$ equals the adjoint action of $\exp(\ad(\li_2(-\sigma e_{\alpha})))$. So we need to find an element of $\widehat{\g}_q$ which plays the role of the (exponential of the) dilogarithm. This is the \emph{$q$-dilogarithm},
\begin{equation*}
{\bf E}(\sigma\hat{e}_{\alpha}) = \sum_{n \geq 0} \frac{(-q^{\frac{1}{2}}\sigma \hat{e}_{\alpha})^n}{(1-q)(1-q^2)\cdots(1-q^n)}
\end{equation*}
(where $\sigma \in \mathfrak{m}_R$). This only involves commuting variables, and is well defined by our assumptions on $R$. The $q$-dilogarithm is in fact a $q$-deformation of $\exp(\ad(\li_2(\sigma e_{\alpha})))$, as shown by the standard rewriting
\begin{equation*}
{\bf E}(\sigma \hat{e}_{\alpha}) = \exp\left(-\sum_{k\geq 1}\frac{\sigma^k\hat{e}_{k\alpha}}{k((-q^{\frac{1}{2}})^k - (-q^{\frac{1}{2}})^{-k})}\right).
\end{equation*}
The $q$-dilogarithm also admits a well known infinite product expansion,
\begin{equation}\label{expansion}
{\bf E}(\sigma\hat{e}_{\alpha}) = \prod_{k \geq 0} \left(1 + q^{k + \frac{1}{2}}\sigma\hat{e}_{\alpha}\right)^{-1}.
\end{equation}
For $\Omega \in \Q$ we introduce automorphisms $\hat{\theta}^{\Omega}[\sigma\hat{e}_{\alpha}]$ of $\widehat{\g}_q$ acting by
\begin{equation*}
\hat{\theta}^{\Omega}[\sigma\hat{e}_{\alpha}](\hat{e}_{\beta}) = \Ad {\bf E}^{\Omega}(\sigma\hat{e}_{\alpha}) (\hat{e}_{\beta}) = {\bf E}^{\Omega}(\sigma\hat{e}_{\alpha}) \hat{e}_{\beta} {\bf E}^{-\Omega}(\sigma\hat{e}_{\alpha}). 
\end{equation*}
We regard $\hat{\theta}[\sigma\hat{e}_{\alpha}]$ as the required quantization of $\theta_{\alpha, 1 + \sigma e_{\alpha}}$. Notice the change of notation from the classical case. This is more practical especially since we will also need to consider the shifted operators $\hat{\theta}[\sigma(-q^{\frac{1}{2}})^n\hat{e}_{\alpha}] = \Ad {\bf E}(\sigma(-q^{\frac{1}{2}})^n\hat{e}_{\alpha})$ for $n \in \Z$. We can now prove the analogue of Lemma \ref{classAction}.  
\begin{lem}\label{qAction} The adjoint action is given by
\begin{equation}\label{qActionEq}
\hat{\theta}^{\Omega}[\sigma(-q^{\frac{1}{2}})^n\hat{e}_{\alpha}](\hat{e}_{\beta}) = \hat{e}_{\beta} \prod^{\bra \alpha, \beta \ket - 1}_{k = 0} \left(1 + (-1)^n q^{k + \frac{n+1}{2}} \sigma\hat{e}_{\alpha}\right)^{\Omega} \prod^{ - 1}_{k = \bra \alpha, \beta \ket} \left(1 + (-1)^n q^{k + \frac{n+1}{2}} \sigma\hat{e}_{\alpha}\right)^{-\Omega}.
\end{equation}
\end{lem}
\noindent It is straightforward to check that \eqref{qActionEq} admits the expected quasi-classical limit \eqref{classActionEq} as $q^{\frac{1}{2}}\to 1$.
\begin{proof} Suppose to start with that $n = 0$ and set $\kappa = \bra\alpha, \beta\ket$. Then $\hat{e}_{\alpha}\hat{e}_{\beta} = q^{\kappa} \hat{e}_{\beta}\hat{e}_{\alpha}$, and so if $f(\hat{e}_{\alpha})$ is a formal power series with coefficients in $\C[q^{\pm \frac{1}{2}}, ((q^n - 1)^{-1})_{n \geq 1}]$, we have 
\begin{equation*}
f(\hat{e}_{\alpha}) \hat{e}_{\beta} = \sum_{i \geq 0} a_i \hat{e}_{i\alpha}\hat{e}_{\beta} = \hat{e}_{\beta} \sum_{i \geq 0} a_i q^{i\kappa}\hat{e}_{i\alpha} = \hat{e}_{\beta} f(q^{\kappa} \hat{e}_{\alpha}).
\end{equation*}
Apply this to $f(\hat{e}_{\alpha}) = {\bf E}^{\Omega}(\sigma \hat{e}_{\alpha})$ to get
\begin{equation*}
\hat{\theta}^{\Omega}[\sigma \hat{e}_{\alpha}](e_{\beta}) = {\bf E}^{\Omega}(\sigma\hat{e}_{\alpha}) \hat{e}_{\beta} {\bf E}^{-\Omega}(\sigma\hat{e}_{\alpha}) = \hat{e}_{\beta}  {\bf E}^{\Omega}(q^{\kappa}\sigma\hat{e}_{\alpha}) {\bf E}^{-\Omega}(\sigma\hat{e}_{\alpha}). 
\end{equation*}
Suppose for a moment that $\kappa \geq 0$ and use the product expansion \eqref{expansion} to get 
\begin{align*}
{\bf E}^{\Omega}(q^{\kappa}\sigma\hat{e}_{\alpha}) &= \prod_{i \geq 0} \left(1 + q^{i + \frac{1}{2}} q^{\kappa} \sigma\hat{e}_{\alpha}\right)^{-\Omega} = \prod_{i \geq \kappa} \left(1 + q^{i + \frac{1}{2}}\sigma\hat{e}_{\alpha}\right)^{-\Omega}\\
&= \prod^{\kappa-1}_{i = 0} \left(1 + q^{i + \frac{1}{2}}\hat{e}_{\alpha}\right)^{\Omega}{\bf E}^{\Omega}(\sigma\hat{e}_{\alpha}),
\end{align*}
and \eqref{qActionEq} follows. For $n \neq 0$ we use similarly
\begin{align*}
{\bf E}(q^{\kappa}(-q^{\frac{1}{2}})^n\hat{e}_{\alpha}) &= \prod_{i \geq 0} \left(1 + (-1)^n q^{i + \frac{n+1}{2}} q^{\kappa} \hat{e}_{\alpha}\right)^{-\Omega} = \prod_{i \geq \kappa} \left(1 + (-1)^n q^{i + \frac{n+1}{2}} \hat{e}_{\alpha}\right)^{-\Omega}\\
&= \prod^{\kappa - 1}_{i = 0} \left(1 + (-1)^n q^{i + \frac{n+1}{2}} \hat{e}_{\alpha}\right)^{\Omega}{\bf E}^{\Omega}((-q^{\frac{1}{2}})^n\hat{e}_{\alpha}).
\end{align*}
The result for $\kappa < 0$ also follows since $\hat{\theta}[\sigma\hat{e}_{\alpha}]$ is an algebra automorphism.
\end{proof}
\noindent We can now introduce the $q$-deformed analogue of the group $\widetilde{\V}_R$. 
\begin{defn} $\mathbb{U}_{\Gamma, R}$ is the completion of the subgroup of $\Aut_{\C(q^{\pm\frac{1}{2}})\otimes_{\C} R} \widehat{\g}_q$ generated by automorphisms of the form $\hat{\theta}^{\Omega}[(-q^{\frac{1}{2}})^n \sigma\hat{e}_{\alpha})]$ (where $\alpha\in\Gamma$, $\sigma\in\m_R$, $\Omega \in Q$, $n \in \Z$), with respect to the $\mathfrak{m}_R$-adic topology. 
\end{defn}
\noindent From now on we assume that $\Gamma$ is $\Z^2$ with its standard antisymmetric bilinear form (as in Lemma \ref{iso}). The factorization \eqref{IntroCommutator} has an analogue in the $q$-deformed case. Suppose that $\alpha_1$ follows $\alpha_2$ in clockwise order.
\begin{lem} Fix positive integers $\ell_1, \ell_2$. Then there exist unique $\Omega_n(k \alpha) \in \mathbb{Q}$ such that
\begin{equation}\label{qCommutator}
[\hat{\theta}^{\ell_1}[\sigma_1\hat{e}_{\alpha_1}], \hat{\theta}^{\ell_2}[\sigma_2\hat{e}_{\alpha_2}]] = \prod^{\to}_{\gamma} \prod_{k \geq 1} \prod_{n \in \Z} \hat{\theta}^{(-1)^n \Omega_n(k\gamma)}[(-q^{\frac{1}{2}})^n \sigma^{k\gamma}\hat{e}_{k\gamma}],
\end{equation}
where $\prod^{\to}$ is a slope ordered product over \emph{primitive, positive} vectors $\gamma = \gamma^1\alpha_1 + \gamma^2\alpha_2$, we set $\sigma^{k\gamma} = \sigma^{k\gamma_1}_1\sigma^{k\gamma_2}_2$ and, for each fixed $k$, $\Omega_n(k\gamma)$ vanishes for all but finitely many $n$.
\end{lem}
\begin{proof} Use the Baker-Campbell-Hausdorff formula and induction on $\gamma^1 + \gamma^2$ (see e.g. \cite{pioline} section 2.1).
\end{proof}
\noindent Our main problem then becomes to find $\hat{\theta}_{\gamma} = \prod_{k \geq 1} \prod_{n \in \Z} \hat{\theta}^{(-1)^n \Omega_n(k\gamma)}[(-q^{\frac{1}{2}})^n \sigma^{k\gamma}\hat{e}_{k\gamma}]$. To compare this with \eqref{VtoTrop} write $\hat{x} = \hat{e}_{1, 0}$, $\hat{y} = \hat{e}_{0,1}$ (so $\hat{x} \hat{y} = q \hat{y}\hat{x}$) and introduce a ``Poincar\'e" Laurent polynomial in $q^{\frac{1}{2}}$, $P(k \gamma)(q) = \sum_{n \in \Z} (-1)^n\Omega_n(k \gamma)(-q^{\frac{1}{2}})^n$. Then the action of 
$$\prod_{k \geq 1} \prod_{n \in \Z} \hat{\theta}^{(-1)^n \Omega_n(k\gamma)}[(-q^{\frac{1}{2}})^n \sigma^{k\gamma}\hat{e}_{k\gamma}]$$ 
can be written as $\hat{x} \mapsto \hat{x} f, \hat{y} \mapsto \hat{y} g$, where $f, g$ are \emph{commutative} power series given by
\begin{align*}
\log f &= \sum_{m\geq 1} q^{-\frac{m^2}{2}\gamma^1\gamma^2} (\sigma_1\hat{x})^{m \gamma^1} (\sigma_2\hat{y})^{m \gamma^2} \left( - \sum_{k | m} \sum^{-1}_{s = - k \gamma^2} q^{\frac{m}{k}s} \right) \hat{c}^{\gamma}_{m},\\
\log g &=\sum_{m\geq 1} q^{-\frac{m^2}{2}\gamma^1\gamma^2} (\sigma_1\hat{x})^{m \gamma^1} (\sigma_2\hat{y})^{m \gamma^2} \left(\sum_{k | m} \sum^{k \gamma^1 - 1}_{s = 0} q^{\frac{m}{k}s} \right) \hat{c}^{\gamma}_{m}
\end{align*}
for coefficients $\hat{c}^{\gamma}_m = \sum_{k | m} \frac{(-q^{\frac{1}{2}})^{\frac{m}{k}}}{\frac{m}{k}} P(k \gamma)(q^{\frac{m}{k}}) \in \Q[q^{\pm\frac{1}{2}}]$. 

\section{Scattering diagrams}\label{scattering}

We adapt the results of \cite{gps} sections 1 and 2 to the $q$-deformed setup. In this section we take $\Gamma$ to be $\Z^2$ with its standard antisymmetric bilinear form (as in Lemma \ref{iso}). 
\begin{defn} A \emph{ray} or \emph{line} for $\mathbb{U}_{\Gamma, R}$ is a pair $(\mathfrak{d}, \hat{\theta}_{\mathfrak{d}})$ where
\begin{enumerate}
\item[$\bullet$] $\mathfrak{d} \subset \Gamma\otimes\R = \R^2$ is a subset which is either of the form $\alpha'_0 + \R_{\geq 0}\alpha_0$ (a ray), or $\alpha'_0 + \R\alpha_0$ (a line), with $\alpha'_0 \in \R^2$, and $\alpha_0\in\Gamma$ positive;
\item[$\bullet$] $\hat{\theta}_{\mathfrak{d}} \in \mathbb{U}_{\Gamma, R}$ is a (possibly infinite) product of elements of the form $\hat{\theta}^{\Omega}[(-q^{\frac{1}{2}})^n\sigma\hat{e}_{k\alpha_0}]$. 
\end{enumerate}
\end{defn}
\noindent If $\d$ is a ray we write $\del \d$ for its initial point, and set $\del \d = \emptyset$ if $\d$ is a line.
\begin{defn} A \emph{scattering diagram} for $\U_{\Gamma, R}$ is a collection of rays and lines $(\mathfrak{d}, \hat{\theta}_{\mathfrak{d}})$ such that for every $k \geq 1$ we have $\hat{\theta}_{\mathfrak{d}} \equiv \operatorname{Id} \operatorname{mod} \mathfrak{m}^k_R$ for all but finitely $\mathfrak{d}$.  
\end{defn}
\noindent The \emph{singular set} of a scattering diagram $\D$ is $\operatorname{Sing}(\D) = \bigcup_{\d \in \D}\del\d\,\cup\,\bigcup_{\dim \d_1 \cap \d_2 = 0} \d_1\cap\d_2$. Let $\pi\!: [0, 1] \to \R^2$ be a smooth path. We say that $\pi$ is \emph{admissible} if $\pi$ misses the singular set $\operatorname{Sing}(\D)$ and is transversal to every $\d \in \D$. We will define a notion of path ordered product $\hat{\theta}_{\pi, \D}$ along $\pi$. Let $k \geq 1$. Then $\pi$ meets (transversely) only finitely many $\d$ with $\hat{\theta}_{\d} \not\equiv \operatorname{Id} \operatorname{mod} \mathfrak{m}^k_R$. We denote this ordered collection by $\d_1, \ldots, \d_s$, and define a partial ordered product $\hat{\theta}^{(k)}_{\pi, \D} = \hat{\theta}^{\eps_1}_{\d_1} \circ \cdots \circ \hat{\theta}^{\eps_s}_{\d_s}$. Here $\eps_i = 1$ if $\{\pi', \alpha_0\}$ is a positive basis of $\R^2$ (where $\alpha_0$ is the direction of $\d_i$), and $\eps_i = -1$ otherwise. Notice that the only ambiguity in $\hat{\theta}^{(k)}_{\pi, \D}$ happens when $\dim \d_i \cap \d_{i+1} = 1$. But then $\hat{\theta}_{\d_i}$ and $\hat{\theta}_{\d_{i+1}}$ commute, so in fact $\hat{\theta}^{(k)}_{\pi, \D}$ is well defined. We then let $\hat{\theta}_{\pi, \D} = \lim_{\to, k} \hat{\theta}^{(k)}_{\pi, \D}$, a well defined element of $\U_{\Gamma, R}$.  We say that a scattering diagram $\D$ is \emph{saturated} if $\hat{\theta}_{\pi, \D} = \operatorname{Id}$ for all admissible, closed paths $\pi$. Two scattering diagrams $\D$, $\D'$ are \emph{equivalent} if $\hat{\theta}_{\pi, \D} = \hat{\theta}_{\pi, \D'}$ whenever $\pi$ is admissible for both. A simple induction argument (adapted e.g. from the proof of \cite{gps} Theorem 1.4) shows that a scattering diagram $\D$ admits a \emph{saturation}: a saturated scattering diagram $\operatorname{S}(\D)$ which is obtained by adding to $\D$ a collection of rays. Moreover $\operatorname{S}(\D)$ is unique up to equivalence. 

Suppose that  the $\alpha_i$ are positive and that $\alpha_1$ follows $\alpha_2$ in clockwise order. Then computing the operators $\hat{\theta}_{\gamma}$ for $[\hat{\theta}^{\ell_1}[\sigma_1\hat{e}_{\alpha_1}], \hat{\theta}^{\ell_2}[\sigma_2\hat{e}_{\alpha_2}]]$ is equivalent to determining $\S(\D)$ for the scattering diagram $\D = \{(\R\alpha_1, \hat{\theta}^{\ell_1}[\sigma_1\hat{e}_{\alpha_1}]), (\R\alpha_2, \hat{\theta}^{\ell_2}[\sigma_2\hat{e}_{\alpha_2}])\}$. To see this choose $\pi$ to be a closed loop with $\pi(0) = (-1, -1)$ and winding once around the origin in clockwise direction. In general computing $\S(\D)$ can be very hard. However there is a special case when saturation is straightforward. 

For $m \in \Z$ we set $[m]_q = \frac{q^{\frac{m}{2}} - q^{-\frac{m}{2}}}{q^{\frac{1}{2}} - q^{-\frac{1}{2}}}$, the usual $q$-number.
\begin{lem}\label{commutator} Suppose $\mathfrak{m}_R$ contains elements $\sigma_1, \sigma_2$ with $\sigma^2_i = 0$, and let 
\begin{equation*}
\D =\{(\R\alpha_1, \hat{\theta}[\sigma_1 \hat{e}_{\alpha_1}]), (\R\alpha_2, \hat{\theta}[\sigma_2 \hat{e}_{\alpha_2}])\}.
\end{equation*}
Suppose that  the $\alpha_i$ are positive and that $\alpha_1$ follows $\alpha_2$ in clockwise order. Then $\S(\D)$ is obtained by adding the single ray
\begin{equation*}
(\R(\alpha_1 +\alpha_2), \hat{\theta}[[\bra \alpha_1, \alpha_2 \ket]_q  \sigma_1 \sigma_2 \hat{e}_{\alpha_1 + \alpha_2}]).
\end{equation*}
\end{lem}
\begin{proof} Since $\sigma^2_i = 0$ we have
\begin{equation*}
\hat{\theta}[\sigma_i \hat{e}_{\alpha_i}] = \Ad {\bf E}(\sigma_i \hat{e}_{\alpha_i}) = \Ad \exp\left(\frac{\sigma_i\hat{e}_{\alpha_i}}{q^{\frac{1}{2}} - q^{-\frac{1}{2}}}\right).
\end{equation*}
Applying the Baker-Campbell-Hausdorff formula we find
\begin{equation*}
\hat{\theta}[\sigma_1 \hat{e}_{\alpha_1}] \hat{\theta}[\sigma_2 \hat{e}_{\alpha_2}] = \Ad\exp\left(\frac{\sigma_1\hat{e}_{\alpha_1}}{q^{\frac{1}{2}} - q^{-\frac{1}{2}}} + \frac{\sigma_2\hat{e}_{\alpha_2}}{q^{\frac{1}{2}} - q^{-\frac{1}{2}}} + \frac{1}{2}[\bra \alpha_1, \alpha_2 \ket]_q\frac{\sigma_1\sigma_2\hat{e}_{\alpha_1 + \alpha_2}}{q^{\frac{1}{2}} - q^{-\frac{1}{2}}}\right).
\end{equation*}
Composing on the left with $\hat{\theta}^{-1}[\sigma_2 \hat{e}_{\alpha_2}]$ gives $\Ad\exp\left(\frac{\sigma_1\hat{e}_{\alpha_1}}{q^{\frac{1}{2}} - q^{-\frac{1}{2}}} + [\bra \alpha_1, \alpha_2 \ket]_q\frac{\sigma_1\sigma_2\hat{e}_{\alpha_1 + \alpha_2}}{q^{\frac{1}{2}} - q^{-\frac{1}{2}}}\right)$. Finally composing on the right with $\hat{\theta}^{-1}[\sigma_1 \hat{e}_{\alpha_1}]$ leaves $\Ad{\bf E}([\bra \alpha_1, \alpha_2 \ket]_q  \sigma_1 \sigma_2 \hat{e}_{w_1\alpha_1 + w_2\alpha_2})$. 
\end{proof}
Let $\D = \{(\d_i, \hat{\theta}_i)\}$ be a scattering diagram which contains only lines. Fix (possibly trivial) factorizations $\hat{\theta}_i = \prod_j \hat{\theta}_{i j}$, where each $\hat{\theta}_{i j}$ is a finite product of $\hat{\theta}^{\Omega}[(-q^{\frac{1}{2}})^n\sigma\hat{e}_{k\alpha}]$. A \emph{perturbation} of $\D$ is a diagram of the form $\widetilde{\D} = \{(\d_i + \beta_{ij}, \hat{\theta}_{i j})\}$ for some $\beta_{ij}\in\R^2$. Suppose now that $\D$ contains only lines through the origin. Then the \emph{asymptotic diagram} $\D'_{\rm as}$ of an arbitrary scattering diagram $\D'$ is defined by replacing $\d = (\alpha'_0 + \R_{\geq 0}\alpha_0, \hat{\theta}_{\d})\in\D'$ by $\d' = (\R_{\geq 0}\alpha_0, \hat{\theta}_{\d})\in\D'_{\rm as}$, and $\d = (\alpha'_0 + \R_{\geq 0}\alpha_0, \hat{\theta}_{\d})\in\D'$ with $\d' = (\R\alpha_0, \hat{\theta}_{\d})\in\D_{\rm as}'$. If $\D$ is a scattering diagram which contains only lines through the origin, and $\widetilde{\D}$ some perturbation, then $\S(\D)$ is equivalent to $\S(\widetilde{\D})_{\rm as}$ (since $\hat{\theta}_{\pi, \S(\widetilde{\D})_{\rm as}} = \operatorname{Id}$ for a sufficiently large simple loop around the origin).

Let $R = \C[[t_1, \ldots, t_n]]$. Fix positive, primitive $\alpha_i \in \Gamma$ (not necessarily distinct) for $i = 1, \ldots, n$. We will write $\boldsymbol\alpha$ for the vector $(\alpha_1, \ldots, \alpha_n)$. A \emph{standard} scattering diagram is one of the form
\begin{equation*}
\D = \{(\R \alpha_i, \prod_{w\geq 1} \hat{\theta}[b_{i w} t_i \hat{e}_{w\alpha_i}]), i = 1, \ldots, n\}
\end{equation*}
where the coefficients $b_{iw} \in \C$ vanish for all but finitely many $w$. By the above discussion computing $[\hat{\theta}^{\ell_1}[t_1\hat{e}_{\alpha_1}], \hat{\theta}^{\ell_2}[t_2\hat{e}_{\alpha_2}]]$ can be regarded as a special case of the problem of saturating a standard scattering diagram (after a perturbation). Passing to $R_k$ we get a truncated diagram $\D_k$, and we can write
\begin{equation*}
\prod_{w\geq 1} \hat{\theta}[b_{i w} t_i \hat{e}_{w\alpha_i}] \equiv \Ad \exp\left(\sum^k_{j = 1}\sum_{w\geq 1} a'_{ijw} \hat{e}_{w\alpha_i} t^j_i\right)
\end{equation*}
for coefficients $a'_{ijw}\in\C[q^{\pm\frac{1}{2}}, ((q^n - 1)^{-1})_{n\geq 1}]$ which vanish for all but finitely many $w$. Following the notation of \eqref{factorization} there are standard factorizations over $\widetilde{R}_k$ given by
\begin{align}\label{stdFactorization}
\nonumber\Ad \exp\left(\sum^k_{j = 1}\sum_{w\geq 1} a'_{ijw} \hat{e}_{w\alpha_i} t^j_i\right) &= \prod_J \prod_w \Ad \exp\left((\#J)! a'_{i(\#J) w}\sum_{s\in J} u_{i s} \hat{e}_{w\alpha_i}\right)\\
&= \prod_J \prod_w \hat{\theta}[(\#J)! a_{i(\#J) w}\sum_{s\in J} u_{i s} \hat{e}_{w\alpha_i}].
\end{align}
Use these to get a perturbation $\widetilde{\D}_k = \{(\R \alpha_i + \beta_{iJw}, \hat{\theta}[(\#J)! a_{i(\#J) w}\sum_{s\in J} u_{i s} \hat{e}_{w\alpha_i}]\}$.
We construct an increasing sequence of diagrams $\widetilde{\D}^{i}_k$, starting from $\widetilde{\D}^0_k = \widetilde{\D}_k$, which stabilize to $\S(\widetilde{\D}_k)$ for $i \gg 1$. We assume inductively that every element of $\D^i_k$ is of the form $(\d, \hat{\theta}[c_{\d} u_{I(\d)} \hat{e}_{\alpha_{\d}}])$, where $c_{\d}\in\C$, $I(\d)\subset \{1, \ldots, n\}\times\{1, \ldots, k\}$ and we set $u_{I(\d)} = \prod_{(i, j)\in I_{\d}} u_{ij}$. We define a \emph{scattering pair} $\{\d_1, \d_2\} \subset \widetilde{\D}^i_k$ to be a pair of lines or rays such that $\d_1, \d_2 \notin \widetilde{\D}^{i-1}_k$, $\d_1\cap\d_2$ is a single point $\alpha'_0$ and $I(\d_1)\cap I(\d_2) = \emptyset$. We let $\widetilde{\D}^{i+1}_k$ be the union of $\widetilde{\D}^i_k$ with all the rays of the form
\begin{equation}\label{outRay}
(\alpha'_0 + \R_{\geq 0}(\alpha_{\d_1} + \alpha_{\d_2}), \hat{\theta}([\bra\alpha_{\d_1}, \alpha_{\d_2}\ket]_q c_{\d_1} c_{\d_2} u_{I(\d_1)} u_{I(\d_2)}\hat{e}_{\alpha_{\d_1} + \alpha_{\d_2}})),
\end{equation}
where $\{\d_1, \d_2\}$ are as above, and we assume (without loss of generality) that the slope of $\d_1$ is smaller than the slope of $\d_2$. For a suitably general initial perturbation $\widetilde{\D}_k$, we can assume that for all $i = 1, \ldots, n k$ and scattering pairs $\{\d_1, \d_2\}$ there is no further $\d \in \widetilde{\D}^i_k$ such that $\d_1 \cap \d_2 \cap \d \neq \emptyset$, and $I(\d_1) \cap I(\d_2) \cap I(\d) = \emptyset$. We will always make this genericity assumption on $\widetilde{\D}_k$ in what follows.   
\begin{lem} The $\widetilde{\D}^i_k$ stabilize to a scattering diagram $\widetilde{\D}^{\infty}_k$ for $i > n k$, and $\widetilde{\D}^{\infty}_k$ is equivalent to $\S(\widetilde{\D}_k)$.
\end{lem}
\begin{proof} If $\d_1, \d_2 \in \widetilde{\D}^i_k \setminus \widetilde{\D}^{i-1}_k$ for $i > n k$, then $I(\d_1) \cap I(\d_2) \neq 0$, so $u_{I(\d_1)} u_{I(\d_2)} = 0$, and $\widetilde{\D}^i_k = \widetilde{\D}^{i + 1}_k$. We set $\widetilde{\D}^{\infty}_k = \widetilde{\D}^{i}_k$ for $i > n k$. If $\widetilde{\D}^{\infty}_k \not\equiv \S(\widetilde{\D}_k)$, we could find a simple loop $\pi$ around a point $\alpha'_0 \in \operatorname{Sing}(\widetilde{\D}^{\infty}_k)$ for which $\hat{\theta}_{\pi, \widetilde{\D}^{\infty}_k} \neq \operatorname{Id}$. By construction this implies that there are two rays $\d_1, \d_2 \in \widetilde{\D}^{\infty}_k$ for which the ray \eqref{outRay} does not belong to $\widetilde{\D}^{\infty}_k$, a contradiction.
\end{proof}
\noindent Let $(\d, \hat{\th}_{\d})$ be an element of some $\widetilde{\D}^{i}_k$. We associate to $\d$ an immersed graph $\Upsilon_{\d} \subset \R^2$, with both bounded and unbounded edges, which is either trivalent or a single line. $\Upsilon_{\d}$ is constructed inductively; if $\d$ is a line then $\Upsilon_{\d} = \d$. Otherwise by construction $\d$ arises uniquely from the scattering of a pair $\{\d_1, \d_2\}\subset \widetilde{\D}^{i-1}_k$, with $\d_1\cap\d_2 = \alpha'_0$. We let 
\begin{equation*}
\Upsilon_{\d} = (\Upsilon_{\d_1}\setminus (\alpha'_0 + \R_{\geq 0}\alpha_{\d_1}))\cup(\Upsilon_{\d_2}\setminus (\alpha'_0 + \R_{\geq 0}\alpha_{\d_2}))\cup (\alpha'_0 + \R_{\geq 0}(\alpha_{\d_1} + \alpha_{\d_2})).
\end{equation*}
The induction makes sense since $\widetilde{\D}^0_k$ contains only lines. Suppose that $\Upsilon_{\d}$ is not a line. Then by construction it contains a finite number of unbounded edges, including the \emph{outgoing} edge $\d$. The other (\emph{incoming}) unbounded edges are all contained in the lines $\d_{i J w}$. By standard arguments we can think of $\Upsilon_{\d}$ as a rational tropical curve. More precisely there exists a unique equivalence class of parametrized rational tropical curves $h\! : \widetilde{\Upsilon}_{\d} \to \R^2$ where $(\widetilde{\Upsilon}_{\d}, w_{\widetilde{\Upsilon}_{\d}})$ is a simply connected trivalent weighted graph (with both bounded and unbounded edges, endowed with its standard topology), $h$ is a proper map with $h(\widetilde{\Upsilon}_{\d}) = \Upsilon_{\d}$, $w_{\widetilde{\Upsilon}_{\d}}(E) = w$ when $E$ is an unbounded edge mapping to a ray contained in $\d_{i J w}$, and the unbounded direction of $h(E)$ is $-\alpha_i$. We will not be very careful in distinguishing $h\! : \widetilde{\Upsilon}_{\d} \to \R^2$ from its image. By construction, we have
\begin{lem}\label{legs} There is a bijective correspondence between elements $(\d, \hat{\theta}_{\d}) \in \widetilde{\D}^i_k$ and rational tropical curves $\Upsilon_{\d}$ whose unbounded edges either coincide with $\d$ or are contained in $\d_{i J w}$ (we prescribe the multiplicity of such an unbounded edge to be $w$).
\end{lem}
\noindent Let $V$ be a vertex of $\widetilde{\Upsilon}_{\d}$, and choose two incident edges $E_1, E_2$. Let $v_1, v_2$ denote the primitive vectors in the direction of $h(E_1), h(E_2)$. The \emph{tropical multiplicity of $h$ at $V$} is defined as $\mu(h, V) = w_{\widetilde{\Upsilon}_{\d}}(E_1)w_{\widetilde{\Upsilon}_{\d}}(E_2)|\det(v_1, v_2)|$. This is well defined by the balancing condition $\sum_i w_{\widetilde{\Upsilon}_{\d}}(E_i) v_i = 0$ where $E_i$ are the incident edges at $V$ and the $v_i$ primitive vectors in the direction of $h(E_i)$ ($i = 1, \ldots, 3$). Then one defines the  \emph{tropical multiplicity of $h$} as $\mu(h) = \prod_V \mu(h, V)$, the product over all trivalent vertices. Abusing notation slightly we will often denote this by $\mu(\Upsilon_{\d})$.
\begin{defn}\label{BGmult}
The \emph{Block-G\"ottsche multiplicity at $V$} is $\mu_q(h, V) = [\mu(h, V)]_q$. Similarly one sets $\mu_q(h) = \prod_V \mu_q(h, V)$. 
\end{defn}
\noindent Using this notion we can reconstruct $\hat{\theta}_{\d}$ from $\Upsilon_{\d}$ with its tropical structure.  
\begin{lem}\label{ScatterToMult} Fix $(\d, \hat{\th}_{\d}) \in \widetilde{\D}^{i}_k$. Let $\alpha_{\rm out} = \sum_{i, J, w} w\alpha_i$, summing over all $i, J, w$ for which there exists an unbounded incoming edge of $\Upsilon_{\d}$ contained in $\d_{iJw}$. Then 
\begin{equation*}
\hat{\theta}_{\d} = \hat{\theta}[\mu_q(\Upsilon_{\d})\left(\prod_{i, J, w} (\# J)! a_{i(\#J)w}\prod_{s\in J} u_{is} \right)\hat{e}_{\alpha_{\rm out}}]. 
\end{equation*}
\end{lem}
\begin{proof} Consider the statement: for two rays $\d_1, \d_2 \in \widetilde{\D}^j_k$, $j < i$, with the slope of $\d_1$ smaller than the slope of $\d_2$, incoming at $V$, and scattering in a ray $\d'$, one has $\mu(\Upsilon_{\d'}, V) = \bra \alpha_{\d_1}, \alpha_{\d_2}\ket$. One checks this for $j = 0$, and by induction it holds for all $j < i$ (the point being of course that $\bra -, - \ket$ is the same as $\det(-, -)$). The statement about $\hat{\theta}_{\d}$ and $\Upsilon_{\d}$ then follows by induction on $i$ and Lemma \ref{commutator}.   
\end{proof}
\noindent  Fix a \emph{weight vector} ${\bf w} = ({\bf w}_1, \ldots, {\bf w}_n)$, where each ${\bf w}_i$ is a collection of integers $w_{ij}$ (for $1 \leq i \leq n$ and $1 \leq j \leq l_i$) such that  $1 \leq w_{i1} \leq w_{i2} \leq \dots \leq w_{i{l_i}}$. For  $1 \leq j \leq l_i$ choose a general collection of parallel lines $\mathfrak{c}_{i j}$ in the direction $\alpha_i$. We attach the weight $w_{ij}$ to the line $\mathfrak{c}_{ij}$, and think of the lines $\mathfrak{c}_{ij}$ as incoming unbounded edges for connected, rational tropical curves $\Upsilon \subset \R^2$. We prescribe that such curves $\Upsilon$ have a single additional outgoing unbounded edge. Let us denote by $\mathcal{S}({\bf w}, \mathfrak{c}_{ij})$ the finite set of such tropical curves $\Upsilon$ (for a general, fixed choice of ends $\mathfrak{c}_{ij}$). 
\begin{defn}
We denote by $\widehat{N}^{\rm trop}_{{\boldsymbol\alpha}}({\bf w}) = \#^{\mu_q} \mathcal{S}({\bf w}, \mathfrak{c}_{ij})$ the $q$-deformed tropical count of curves $\Upsilon$ as above, i.e. the number of elements of $\mathcal{S}({\bf w}, \mathfrak{c}_{ij})$ counted with the Block-G\"ottsche multiplicity $\mu_q$. 
\end{defn}
\noindent By the results of \cite{mik}, $\widehat{N}^{\rm trop}_{{\boldsymbol\alpha}}({\bf w})$ is independent of a general choice of $\mathfrak{c}_{ij}$ (see also the Appendix), so in particular it makes sense to drop $\mathfrak{c}_{ij}$ from the notation. 
\begin{cor}\label{ScatterToTrop} Let $(\d, \hat{\theta}_{\d}) \in \S(\D)$. Then
\begin{equation*}
\hat{\theta}_{\d} = \Ad\exp\left(\sum_{\bf w} \sum_{\bf k} \frac{\widehat{N}^{\rm trop}_{{\boldsymbol\alpha}}({\bf w})}{|\Aut({\bf w},{\bf k})|}\left(\prod a_{i k_{ij} w_{ij}}t^{k_{ij}}_i\right)\frac{\hat{e}_{\sum_i |{\bf w}_i| \alpha_i}}{q^{\frac{1}{2}} - q^{-\frac{1}{2}}}\right),
\end{equation*}
where the first sum is over weight vectors ${\bf w} = ({\bf w}_1, \ldots, {\bf w}_n)$ for which $\sum_i |{\bf w}_i| \alpha_i \in \d$, the second sum is over collections ${\bf k} = ({\bf k_1}, \ldots, {\bf k_n})$ of vectors ${\bf k}_i$ with the same length as ${\bf w}_i$, such that $k_{ij} \leq k_{i(j+1)}$ if $w_{ij} = w_{i(j+1)}$, and $\Aut({\bf w},{\bf k})$ is the product of all stabilizers $\Aut({\bf w}_i, {\bf k}_i)$ in the symmetric group. 
\end{cor}
\begin{proof} It is enough to prove the statement modulo $(t^{k+1}_1, \ldots, t^{k+1}_n)$ for all $k$. By the above discussion we know that $\hat{\theta}_{\d}$ equals the product of operators $\hat{\theta}_{\d'}$ for $\d' \in \widetilde{\D}^{\infty}_k$ parallel to $\d$. By Lemma \ref{legs} each $\d'$ corresponds to a unique set $\{\d_{i J w}\}\subset \widetilde{\D}^0_k$. The indices $w$ occurring in $\{\d_{i J w}\}$ for $i = 1, \ldots, n$ then define the collection of weight vectors ${\bf w}$. The indices $\{J\}$ on the other hand define a collection of subsets of $\{1, \ldots, n\}\times\{1, \ldots, k\}$; their cardinalities determine the collection of vectors ${\bf k}$. By Lemma \ref{ScatterToMult} the product of all $\hat{\theta}_{\d'}$ for fixed $\{\d_{i J w}\}$ equals $\Ad\exp\left(\#^{\mu_q}\mathcal{S}({\bf w}, \{\d_{i J w}\})\left(\prod k_{ij}! a_{i k_{ij} w_{ij}} \prod_J u_{J} \right)(q^{\frac{1}{2}} - q^{-\frac{1}{2}})^{-1}\hat{e}_{\sum_i |{\bf w}_i| \alpha_i}\right)$. The key point is that by the first statement in Theorem 1 of \cite{mik} $\#^{\mu_q}\mathcal{S}({\bf w}, \{\d_{i J w}\})$ only depends on $\{\d_{i J w}\}$ through ${\bf w}$ and equals $\widehat{N}^{\rm trop}_{\boldsymbol \alpha}({\bf w})$. For the reader's convenience a sketch of the proof is provided in the Appendix. From here the argument proceeds as in \cite{gps} Theorem 2.8.
\end{proof}

\section{Some $q$-deformed Gromov-Witten invariants}\label{mpsSection}
We follow the notation of the previous section. Consider the scattering diagram over $\C[[s, t]]$ given by $\D' = \{(\R \alpha_1, \hat{\theta}^{\ell_1}[s \hat{e}_{\alpha_1}]), (\R \alpha_2, \hat{\theta}^{\ell_2}[t \hat{e}_{\alpha_2}])\}$ for $\ell_i \in \N$. The saturation $\S(\D')$ can be recovered from $\S(\D)$ for $\D$ the diagram over $\C[[s_1, \ldots, s_{\ell_1}, t_1, \ldots, t_{\ell_2}]]$ given by 
\begin{equation*}
\D = \{(\R\alpha_1, \hat{\theta}[s_1 e_{\alpha_1}]), \ldots, (\R\alpha_1, \hat{\theta}[s_{\ell_1} e_{\alpha_1}]), (\R \alpha_2, \hat{\theta}[t_1 \hat{e}_{\alpha_2}]), \ldots, (\R \alpha_2, \hat{\theta}[t_{\ell_2} \hat{e}_{\alpha_2}])\}.
\end{equation*}
Recall
\begin{equation*}
\hat{\theta}[s_i \hat{e}_{\alpha_1}] = \Ad \exp\left(-\sum^k_{j = 1}\frac{s^j_i \hat{e}_{j\alpha_1}}{j((-q^{\frac{1}{2}})^j - (-q^{\frac{1}{2}})^{-j})}\right),
\end{equation*}
so in the notation of equation \eqref{stdFactorization} we have
\begin{equation*}
a'_{i j j} = \frac{(-1)^j}{j((-q^{\frac{1}{2}})^j - (-q^{\frac{1}{2}})^{-j}))},\quad a_{ijj} = ((-q^{\frac{1}{2}}) - (-q^{-\frac{1}{2}}))a'_{ijj} = \frac{(-1)^{j - 1}}{j [j]_q}
\end{equation*}
while all other $a_{i j w}$ vanish. It follows that in the notation of Corollary \ref{ScatterToTrop}, the coefficients $a_{i k_{ij} w_{ij}}$ are nonvanishing precisely when the sequences $k_{ij}$ and $w_{ij}$ coincide, in which case one has
\begin{equation*}
a_{i w_{ij} w_{ij}} = \frac{(-1)^{w_{ij} - 1}}{w_{ij} [w_{ij}]_q}.
\end{equation*}
The same computation holds for $\hat{\theta}[t_i \hat{e}_{\alpha_2}]$. For the diagram $\D$ the $(\ell_1 + \ell_2)$-tuple of weight vectors ${\bf w}$ required by Corollary \ref{ScatterToTrop} can be parametrized in a more convenient way. We first fix a pair of ordered partitions $({\bf P}_1, {\bf P}_2)$ of length $\ell_1$, $\ell_2$ respectively. The part ${\bf P}_{1i}$ determines the size of a weight vector corresponding to $(\R\alpha_1, \hat{\theta}[s_i e_{\alpha_1}])$, and similarly for ${\bf P}_{2i}$. So we can enumerate instead in terms of just a pair of weight vectors $({\bf w}_1, {\bf w}_2)$, with ${\bf w}_i = (w_{i1} , \ldots , w_{i l_i})$ such that $1 \leq w_{i1} \leq w_{i2} \leq \dots \leq w_{i{l_i}}$, plus a pair of \emph{compatible set partitions} $(I_{1, \bullet}, I_{2, \bullet})$ of the index sets $\{1, \ldots, l_{1}\}$, $\{1, \ldots, l_{2}\}$, that is set partitions $I_{i, \bullet}$ for which $\sum_{s \in I_{i j}} w_{i s} = {\bf P}_{ij}$. Let us denote by $\#\{I_{i, \bullet}, {\bf P}_i | {\bf w}_i\}$ the number of set partitions of ${\bf w}_i$ which are compatible with ${\bf P}_i$, and introduce the \emph{q-deformed ramification factors}
\begin{equation}\label{qRam}
\widehat{R}_{{\bf P}_i | {\bf w}_i, q} = \prod_{j} \frac{(-1)^{w_{ij}-1}}{w_{ij}[w_{ij}]_q} \#\{I_{i, \bullet}, {\bf P}_i | {\bf w}_i\}.
\end{equation}
Then we may apply Corollary \ref{ScatterToTrop} to see that the operators appearing in the saturation of $\D$ are given by
\begin{equation}\label{qVtoTrop}
\hat{\theta}_{a_1\alpha_1 + a_2\alpha_2} = \Ad\exp\left(\sum_{|{\bf P}_1| = k a_1} \sum_{|{\bf P}_2| = k a_2} \sum_{{\bf w}} \prod^2_{i = 1} \frac{\widehat{R}_{{\bf P}_i | {\bf w}_i}}{|\Aut({\bf w}_i)|} \widehat{N}^{\rm trop}_{(\alpha_1, \alpha_2)}({\bf w}) s^{{\bf P}_1} t^{{\bf P}_2} \frac{\hat{e}_{k(a_1\alpha_1 + a_2\alpha_2)}}{q^{\frac{1}{2}}- q^{-\frac{1}{2}}} \right).
\end{equation}
By specialization we finally obtain the analogue of the basic GPS identity \eqref{VtoTrop}. 
\begin{prop}\label{qVtoTropSpec} The operators in the slope ordered expansion for $[\hat{\theta}^{\ell_1}[t \hat{e}_{\alpha_1}], \hat{\theta}^{\ell_2}[t \hat{e}_{\alpha_2}]]$ are obtained by setting $s_i = t_j = t$ in \eqref{qVtoTrop} for $i, j = 1, \ldots, n$.
\end{prop}
\noindent By comparing \eqref{qVtoTrop} with the formulae \eqref{gpsThm}-\eqref{gpsThm2} it is natural to introduce a class of putative $q$-deformed Gromov-Witten invariants.
\begin{defn}\label{refinedGW} We define a $q$-deformation of the invariant $N[({\bf P}_1, {\bf P}_2)]$ by
\begin{equation*}
\widehat{N}[({\bf P}_1, {\bf P}_2)] = \sum_{{\bf w}} \prod^2_{i = 1} \frac{\widehat{R}_{{\bf P}_i | {\bf w}_i}}{|\Aut({\bf w}_i)|} \widehat{N}^{\rm trop}_{(\alpha_1, \alpha_2)}({\bf w}).
\end{equation*}
\end{defn}
\noindent The existence of a $q$-deformation of $N[({\bf P}_1, {\bf P}_2)]$ is natural from the point of view of \cite{gps} section 6, which shows how these invariants may be regarded as Gromov-Witten invariants of a log Calabi-Yau surface. For the same reason the invariants admit a conjectural BPS structure described in ibid. sections 6.2 and 6.3. In the special case when $(|{\bf P}_1|, |{\bf P}_2|)$ is primitive one forms the generating series
\begin{equation*}
N(\tau) = \sum_{k\geq 1} N[(k{\bf P}_1, k{\bf P}_2)] \tau^k,
\end{equation*}
where $k {\bf P}_i$ is the ordered partition with parts $k {\bf P}_{ij}$. Then the BPS invariants $n_k[({\bf P}_1, {\bf P}_2)]$ are defined through the equality
\begin{equation}\label{BPS}
N(\tau) = \sum_{k\geq 1}n_k[({\bf P}_1, {\bf P}_2)]\sum_{d\geq 1} \frac{1}{d^2}\binom{d(k - 1) - 1}{d - 1}\tau^{dk}. 
\end{equation}
As usual the BPS invariants $n_k[({\bf P}_1, {\bf P}_2)]$ are expected to be integers (a special case of \cite{gps} Conjecture 6.2).\\ 

\begin{exa}
Let $({\bf P}_1, {\bf P}_2) = (1, 1 + 1)$. Then $N[(1,1 + 1)] = 1$, and a lengthy computation using \eqref{TropToGW} shows that $N[(2, 2 + 2)] = -\frac{1}{4}$, $N[(3, 3 + 3)] = \frac{1}{9}$. Thus for the first few BPS invariants we find $n_1[(1, 1 + 1)] = 1$, $n_2[(1, 1 +1)] = N[(2, 2 + 2)] + \frac{1}{4} N[(1, 1 + 1)] = 0$, $n_3[(1, 1+1)] = N[(3, 3 + 3)] - \frac{1}{9}N[(1,1 + 1)] = 0$. In this particular example, using the quiver techniques developed in \cite{reinweist}, one could show that indeed $n_k[(1, 1 + 1)] = 0$ for $k > 1$. Here is a sketch of the proof: using the (very simple) representation theory of complete bipartite quivers with one source and at most two sinks, one can check that in this case \cite{reinweist} Lemma 10.1 applies, and deduce the equality $N[(k,  k + k)] = \frac{(-1)^{k-1}}{k^2}$ for $k \geq 1$. Using the definition of $n_k[(1, 1 + 1)]$ and assuming inductively that $n_{i}[(1, 1 + 1)] = \delta_{1i}$, we find $n_{k}[(1, 1 + 1)] = N[(k, k + k)] + \frac{(-1)^k}{k^2}N[1, 1 + 1] = 0$.  
\end{exa}

\noindent For primitive $(|{\bf P}_1|, |{\bf P}_2|)$, $N[({\bf P}_1, {\bf P}_2)]$ equals the BPS invariant $n_1[({\bf P}_1, {\bf P}_2)]$. As a nontrivial consequence there is a more straightforward candidate for $\widehat{N}[({\bf P}_1, {\bf P}_2)]$. Consider the complete bipartite quiver $\mathcal{K}(\ell_1, \ell_2)$ endowed with its natural notion of stability (\cite{us} section 4). The pair $({\bf P}_1, {\bf P}_2)$ induces a dimension vector for $\mathcal{K}(\ell_1, \ell_2)$, and we denote by $\mathcal{M}({\bf P}_1, {\bf P}_2)$ the resulting moduli space of stable representations. By a result of Reineke and Weist (\cite{reinweist} Corollary 9.1) we have
\begin{equation*}
N[({\bf P}_1, {\bf P}_2)] = \chi(\mathcal{M}({\bf P}_1, {\bf P}_2)).
\end{equation*}
Accordingly a natural choice would be to define a $q$-deformed Gromov-Witten invariant as the Poincar\'e polynomial $P(\mathcal{M}({\bf P}_1, {\bf P}_2))(q) = \sum_j b_j(\mathcal{M}({\bf P}_1, {\bf P}_2)) q^{\frac{j}{2}}$ (recall that in fact the odd Betti numbers $b_{2j+1}$ vanish). However it makes more sense to have a notion which is invariant under the change of variable $q \mapsto q^{-1}$. Thus we set
\begin{equation*}
\widehat{N}'[({\bf P}_1, {\bf P}_2)] = \widehat{P}(\mathcal{M}({\bf P}_1, {\bf P}_2))(q) := q^{-\frac{1}{2}\dim \mathcal{M}({\bf P}_1, {\bf P}_2)} P(\mathcal{M}({\bf P}_1, {\bf P}_2))(q). 
\end{equation*}
In the notation of \cite{mozrein} $\widehat{P}(\mathcal{M}({\bf P}_1, {\bf P}_2))(q) = P([\mathcal{M}({\bf P}_1, {\bf P}_2)]_{\rm vir})(q)$, where for $X$ an irreducible smooth algebraic variety one sets $[X]_{\rm vir} = q^{-\frac{1}{2}\dim X} [X]$, a Laurent polynomial with coefficients in the Grothen\-dieck ring of varieties.  
\begin{thm}\label{comparison} The two choices of quantization coincide: $\widehat{N}'[({\bf P}_1, {\bf P}_2)] = \widehat{N}[({\bf P}_1, {\bf P}_2)]$.
\end{thm}
\begin{proof} We will reduce the statement to a representation-theoretic formula due to Manschot, Pioline and Sen \cite{mps}. A \emph{refinement} of $({\bf P}_1, {\bf P}_2)$ is a pair of sets of integers $(k^1, k^2) = (\{k^1_{w, i}\}, \{k^2_{w, j}\})$ such that for $i=1,\ldots,\ell_1$ and $j=1,\ldots, \ell_2$ we have ${\bf P}_{1 i}=\sum_w wk_{w, i}^1,\,{\bf P}_{2 j}=\sum_w wk_{w, j}^2$. We denote refinements by $(k^1,k^2)\vdash ({\bf P}_1,{\bf P}_2)$, and write $m_w(k^i) = \sum_j k^i_{w, j}$. A fixed refinement $k^i$ induces a weight vector ${\bf w}(k^i) = (w_{i1}, \ldots, w_{it_i})$ of length $t_i = \sum_w m_w(k^i)$, by $w_{ij}=w\text{ for all }j=\sum_{r=1}^{w-1}m_r(k^i)+1,\ldots,\sum_{r=1}^{w}m_r(k^i)$. By a combinatorial argument contained in the proof of \cite{us} Lemma 4.2 we may rearrange Definition \ref{refinedGW} as
\begin{equation*} 
\widehat{N}[({\bf P}_1, {\bf P}_2)] = \sum_{(k_1,k_2)\vdash ({\bf P}_1,{\bf P}_2)}\prod^2_{i=1}\prod^{\ell_i}_{j=1}\prod_{w}\frac{(-1)^{k^i_{w,j}(w-1)}}{k^i_{w,j}!w^{k^i_{w,j}}[w]^{k^i_{w,j}}_q}\widehat{N}^{\rm trop}_{(\alpha_1, \alpha_2)}({\bf w}(k^1), {\bf w}(k^2)).
\end{equation*}
From now we follow closely the treatment in \cite{us} section 4; in particular we will formulate the MPS result using the infinite bipartite quiver $\mathcal{N}$ introduced there, with vertices $\mathcal{N}_0=\{i_{(w,m)}\mid (w,m)\in\N^2\}\cup \{j_{(w,m)}\mid (w,m)\in\N^2\}$ 
and arrows $\mathcal{N}_1=\{\alpha_1,\ldots,\alpha_{w\cdot w'}:i_{(w',m')}\rightarrow j_{(w,m)},\forall\,w,w',m,m'\in\N\}$. The quiver $\mathcal{N}$ comes with a notion of stability in terms of a slope function $\mu$ (there is no possible confusion with the tropical multiplicity $\mu$ since the latter will not appear in this proof). Recall that $(k^1,k^2)$ induces a \emph{thin} (i.e. type one) dimension vector $d(k^1,k^2)$ for $\mathcal{N}$, so we get a moduli space of stable \emph{abelian} representations $\mathcal{M}_{d(k^1,k^2)}(\mathcal{N})$. Following the argument leading to equation (1) of \cite{us}, and after rearranging to pass from $P$ to $\widehat{P}$, the MPS formula for Poincar\'e polynomials in this setup can be expressed as
\begin{equation*}
\widehat{P}(\mathcal{M}({\bf P}_1,{\bf P}_2))(q) = \sum_{(k^1,k^2)\vdash({\bf P}_1,{\bf P}_2)}\prod^2_{i=1}\prod^{\ell_i}_{j=1}\prod_{w}\frac{(-1)^{k^i_{w, j}(w-1)}}{k^i_{w, j}!w^{k^i_{w, j}}[w]^{k^i_{w, j}}_q}\widehat{P}(\mathcal{M}_{d(k^1,k^2)}(\mathcal{N}))(q).
\end{equation*}
Indeed in the general case (for $(|{\bf P}_1|, |{\bf P}_2|)$ not necessarily primitive) one can rewrite the MPS formula as 
\begin{equation*}
\frac{[R^{\rm sst}_{({\bf P}_1,{\bf P}_2)}(K(\ell_1, \ell_2))]_{\rm vir}}{[\operatorname{GL}({\bf P}_1,{\bf P}_2)]_{\rm vir}} = \sum_{(k^1,k^2)\vdash({\bf P}_1,{\bf P}_2)}\prod^2_{i=1}\prod^{\ell_i}_{j=1}\prod_{w}\frac{(-1)^{k^i_{w, j}(w-1)}}{k^i_{w, j}!w^{k^i_{w, j}}[w]^{k^i_{w, j}}_q}\frac{[R^{\rm sst}_{d(k^1, k^2)}(\mathcal{N})]_{\rm vir}}{[(\C^*)^{|k^1| + |k^2|}]_{\rm vir}}
\end{equation*}
where we have denoted by $R^{\rm sst}(-)$ the semistable loci, and by $\operatorname{GL}({\bf P}_1,{\bf P}_2)$ the usual basechange group corresponding to a dimension vector; this is explained in \cite{mozrein} section 8.1. The claim of the Lemma then follows from the identity 
\begin{equation}\label{TropToPoin}
\widehat{P}(\mathcal{M}_{d(k^1,k^2)}(\mathcal{N}))(q) = \widehat{N}^{\rm trop}_{(\alpha_1, \alpha_2)}({\bf w}(k^1), {\bf w}(k^2)).
\end{equation}
To prove this let $\mathcal{Q} \subset \mathcal{N}$ denote the finite subquiver which is the support of $(k^1, k^2)$. Consider the lattice $\widetilde{\Gamma} = \Z\mathcal{Q}_0$ endowed with the bilinear form $\bra -, -\ket$ which is the antisymmetrization of the Euler form of $\mathcal{Q}_0$. We will write $\widetilde{\Gamma}^+_{\mu}$ for the subsemigroup of dimension vectors of slope $\mu$. Let $R = \C[[t_{i_{(w',m')}}, t_{j_{(w, m)}}]]$. We work in the group $\U_{\widetilde{\Gamma}, R}$ and consider the product of operators
\begin{equation}\label{quiverProduct}
\prod_{j_{(w, m)}\in \mathcal{Q}_0} \hat{\theta}[t_{j_{(w, m)}}\hat{e}_{j_{(w, m)}}]\prod_{i_{(w',m')}\in \mathcal{Q}_0} \hat{\theta}[t_{i_{(w',m')}}\hat{e}_{i_{(w',m')}}].
\end{equation}
By \cite{reineke} Lemma 4.3, \eqref{quiverProduct} can be expressed as an ordered product $\prod^{\leftarrow}_{\mu \in \Q} \Ad \widetilde{P}_{\mu}$ where $\widetilde{P}_{\mu} \in \hat{\g}_q$ is an element of the form $\sum_{d \in \widetilde{\Gamma}^+_{\mu}} \tilde{p}_d(q) t^d \hat{e}_{d}$ for some $\tilde{p}_d(q)\in\Q(q)$. The $\widetilde{P}_{\mu}$ are characterized in terms of the Harder-Narasimhan recursion given in \cite{reineke} Definition 4.1 and the Remark following it. By the definition of the $\widetilde{P}_{\mu}$, using that $(|{\bf P}_1|, |{\bf P}_2|)$ is primitive and $d(k^1, k^2)$ is thin, one shows that the first nontrivial term in $\widetilde{P}_{\mu(d(k^1, k^2))}$ is $\tilde{p}_{d(k^1, k^2)}(q)t^{d(k^1, k^2)}\hat{e}_{d(k^1, k^2)}$ (as a term of smaller degree would imply the existence of a subrepresentation of the $d(k^1, k^2)$-dimensional representation having the same slope). By the remark following \cite{reineke} Proposition 4.5 we have in fact\footnote{To compare with the results of \cite{reineke} section 4 one must take into account the different convention for the $q$-deformed product. Our functions $\widetilde{P}_{\mu}$, $\tilde{p}_d(q)$ are precisely those which arise when one replaces the product of \cite{reineke} Definition 3.1 with our \eqref{qProduct} (beware that in ibid. the notation $\bra -, - \ket$ denotes the Euler form, \emph{not} its antisymmetrization).}
\begin{equation*}
\tilde{p}_{d(k^1,k^2)}(q) = (q^{\frac{1}{2}} - q^{-\frac{1}{2}})^{-1}\widehat{P}(\mathcal{M}_{d(k^1,k^2)}(\mathcal{N}))(q).
\end{equation*} 
On the other hand by an argument contained in the proof of \cite{us} Proposition 4.3 we can find a change of variables which preserves slopes and reduces the calculation of the $\widetilde{P}_{\mu}$ for \eqref{quiverProduct} to the problem of saturating a scattering diagram for $\U_{\Gamma, R}$ with $\Gamma = \Z^2$ (endowed with its standard antisymmetric bilinear form). We can then combine the rest of the proof of \cite{us} Proposition 4.3 with Lemma \ref{ScatterToMult} above to compute the first nontrivial term in $\widetilde{P}_{\mu(k^1, k^2)}$ as $(q^{\frac{1}{2}} - q^{-\frac{1}{2}})^{-1}\widehat{N}^{\rm trop}_{(\alpha_1, \alpha_2)}({\bf w}(k^1), {\bf w}(k^2))$. Notice that there are no ramification factors precisely because $d(k^1, k^2)$ is thin and coprime, and we are computing the first nontrivial term. Matching the two answers gives \eqref{TropToPoin}.
\end{proof}
\begin{exa}
Consider the case $({\bf P}_1, {\bf P}_2) = (1 + 1, 1 + 1 + 1)$. We can compute the Poincar\'e polynomial e.g. by applying \cite{reinweist} Theorem 9.2 (see also \eqref{explicit} below), finding $P(\mathcal{M}({\bf P}_1, {\bf P}_2))(q) = 1 + 4 q + q^2$. On the other hand the only compatible weight vector ${\bf w}$ is the trivial refinement $({\bf w}_1, {\bf w}_2) = ((1 , 1), (1 , 1 , 1))$, with $\prod^2_{i = 1} |\Aut({\bf w}_i)| = 2 \cdot 3!$ and also (according to \eqref{qRam}) $\prod^2_{i = 1} \widehat{R}_{{\bf P}_i | {\bf w}_i} = 12$. Thus by Definition \ref{refinedGW} we have $\widehat{N}[({\bf P}_1, {\bf P}_2)] = \widehat{N}^{\rm trop}_{(\alpha_1, \alpha_2)}({\bf w})$ with $\alpha_1 = (1, 0), \alpha_2 = (0,1)$. For a suitable configuration of lines $\d_{ij}$ ($1\leq i \leq 2$, $1\leq j \leq 3$), the set $\mathcal{S}({\bf w}, \d_{ij})$ of rational tropical curves with ends weighted by ${\bf w}$ and lying on $\d_{ij}$ contains a curve of multiplicity $[1]_q [2]^2_q = q^{-1} + 2 + q$, and two distinct curves of multiplicity $[1]_q = 1$ (these curves are depicted in \cite{us} section 6.3.1). The tropical count therefore equals $q^{-1} + 4 + q = q^{-1} \widehat{P}(\mathcal{M}({\bf P}_1, {\bf P}_2))(q)$.  
\end{exa}

\begin{exa}
As an example which actually involves a nontrivial $q$-deformed ramification factor we consider $({\bf P}_1, {\bf P}_2) = (1 + 1, 1 + 2)$. In this case the relevant weight vectors are ${\bf w}' = ((1, 1), (1, 2))$, with $\prod^2_{i = 1} |\Aut({\bf w}'_i)| = 2$, $\prod^2_{i = 1} \widehat{R}_{{\bf P}_i | {\bf w}'_i} = -\frac{1}{2[2]_q}$, and ${\bf w}'' = ((1, 1), (1, 1,1))$, with $\prod^2_{i = 1} |\Aut({\bf w}''_i)| = 2\cdot 3!$, $\prod^2_{i = 1} \widehat{R}_{{\bf P}_i | {\bf w}''_i} = 6$. Setting $\alpha_1 = (1, 0), \alpha_2 = (0,1)$, we already computed $\widehat{N}^{\rm trop}_{(\alpha_1, \alpha_2)}({\bf w}'') = q^{-1} + 4 + q$. For a suitable configuration $\d_{ij}$ the set $\mathcal{S}({\bf w}', \d_{ij})$ containts a curve of multiplicity $[4]_q$, as well as two curves of multiplicity $[2]_q$. According to Definition \ref{refinedGW} we compute $\widehat{N}[({\bf P}_1, {\bf P}_2)] = -\frac{1}{2[2]_q} ([4]_q + 2[2]_q) + \frac{1}{2}(q^{-1} + 4 + q) = 1$. On the quiver side, one can in fact prove that in this case $\mathcal{M}({\bf P}_1, {\bf P}_2)$ is just a point (\cite{reinweist} section 5). 
\end{exa}

\begin{rmk} Theorem 4.1 in \cite{us} shows that the MPS formula (going from nonabelian to abelian representations) for $\chi(\mathcal{M}({\bf P}_1, {\bf P}_2))$ and primitive $(|{\bf P}_1|, |{\bf P}_2|)$ is dual to a standard degeneration formula in Gromov-Witten theory (going from incidence conditions to tangency conditions). The proof of Lemma \ref{comparison} gives another interpretation of the MPS formula at the level of Poincar\'e polynomials, as a compatibility condition between two natural $q$-deformations of the invariant $\widehat{N}[({\bf P}_1, {\bf P}_2)]$.   
\end{rmk}

\noindent We close this section with some further remarks concerning the $q$-deformed Gromov-Witten invariants $\widehat{N}[({\bf P}_1, {\bf P}_2)]$.

\subsection{Explicit formula} For primitive $(|{\bf P}_1|, |{\bf P}_2|)$, an explicit (and complicated) formula for the $q$-deformed Gromov-Witten invariant $\widehat{N}[({\bf P}_1, {\bf P}_2)]$ follows immediately from the result for quiver Poincar\'e polynomials \cite{reinweist} Theorem 9.2 (taking into account the dimension formula of ibid. Theorem 5.1): we have
\begin{align}\label{explicit}
\nonumber \widehat{N}[({\bf P}_1, {\bf P}_2)] = &(q - 1) q^{- \frac{1}{2}( 1 - \sum_{i} p^2_{1i} - \sum_{j} p^2_{2j} + \sum_{k, l} p_{1k} p_{2l})} \sum (-1)^{s - 1} q^{\sum_{r \leq s} a_r b_s - \sum_k p^r_{1k} p^s_{1 k} - \sum_l p^r_{2k} p^s_{2 k}}\\ &\prod^t_{r = 1} \left[\prod_k \prod^{p^r_{1k}}_{j = 1} (1 - q^{-j}) \prod_l \prod^{p^r_{2l}}_{j = 1}(1 - q^{-j})^{-1}\right],
\end{align}
where the sum runs over all decompositions ${\bf P}_i = {\bf P}^{(1)}_i + \cdots {\bf P}^{(t)}_i$ for $i = 1, 2$ into ordered partitions ${\bf P}^{(r)}_i = p^r_{i1} + \cdots + p^{r}_{i l_{i}}$ such that setting $(a_r, b_r) = (|{\bf P}^{r}_1|, |{\bf P}^{r}_2|)$ we have $(a_r, b_r)  \neq (0, 0)$ for all $r = 1, \ldots, t$ and 
\begin{equation*}
\frac{b_1 + \cdots + b_r}{a_1 + \cdots + a_r} > \frac{b}{a}
\end{equation*}
for $r < t$. Notice that it is not even clear from this formula that $\widehat{N}[({\bf P}_1, {\bf P}_2)]$ is in fact a symmetric Laurent polynomial.

\subsection{Relation to refined Severi degrees}  When the underlying toric surface is $\PP^2$, some of the $q$-deformed invariants $\widehat{N}[({\bf P}_1, {\bf P}_2)]$ have already been studied from a different point of view in \cite{vivek}, where they are called \emph{refined Severi degrees}. Suppose that $(a, b) = (1, 1)$ and $({\bf P}^d_1, {\bf P}^d_2)$ is the unique pair of partitions of type one, ${\bf P}^d_{ij} = 1$, such that $(|{\bf P}^d_1|, |{\bf P}^d_2|) = (d, d)$, 
\begin{equation*}
({\bf P}^d_1, {\bf P}^d_2) = (\overbrace{1 + \cdots + 1}^{d \textrm{ times}}, \overbrace{1 + \cdots + 1}^{d \textrm{ times}}).
\end{equation*}
Then the non-deformed invariant $N[({\bf P}_1, {\bf P}_2)]$ is effectively enumerating rational curves in $\PP^2$ with geometric genus $0$ and degree $d$, passing through $d$ prescribed general points on the line $D_1$, respectively $d$ prescribed general points on $D_2$, and which are maximally tangent to $D_{\rm out}$. Such a number is usually called a \emph{relative Severi degree}. We follow closely \cite{vivek} section 5 as a reference for Severi degrees; in standard Severi degree notation one would write
\begin{equation*}
N[({\bf P}_1, {\bf P}_2)] = N^{d, \delta_d}(0, e_{d}) \, \textrm{ for } \, \delta_d = \binom{d - 1}{2},
\end{equation*} 
as we now briefly recall. The invariant $N^{d, \delta_d}(0, e_{d})$ is an instance of the general relative Severi degrees $N^{d, \delta}(\alpha, \beta)$, counting curves in $\PP^2$ with degree $d$, geometric genus 
\begin{equation*}
g = \binom{d - 1}{2} - \delta
\end{equation*}
and tangency conditions along $D^{\rm out}$ encoded by the two ordered partitions $\alpha, \beta$, which moreover pass through $2 d + g - 1 + |\beta|$ additional prescribed points. For $k \geq 1$, there are $\alpha_k$ \emph{specified} tangency points of order $k$ on $D^{\rm out}$, as well as $\beta_k$ \emph{unspecified} points of order $k$, with the constraint
\begin{equation*}
\sum_i i \alpha_i + \sum_j j \beta_j = d. 
\end{equation*} 
The partition $e_k$ is a singleton at place $k$. The invariants $N^{d, \delta}(\alpha, \beta)$ are uniquely determined by the fundamental Caporaso-Harris recursion, \cite{vivek} Recursion 70.

Let us now turn to $q$-deformations. In \cite{vivek} section 5.1, the authors introduce a $q$-deformation $\overline{N}^{d, \delta}(\alpha, \beta)$ of the general invariants $N^{d, \delta}(\alpha, \beta)$ (for arbitrary genera), via a $q$-deformation of the Caporaso-Harris recursion, namely \cite{vivek} equation 17:
\begin{equation}\label{capharris}
\overline{N}^{d, \delta}(\alpha, \beta) = \sum_{k : \beta_k > 0} [k]_q \overline{N}^{d, \delta}(\alpha + e_k, \beta - e_k) + \sum_{\alpha', \beta'} \prod_i [i]^{\beta'_i - \beta_i}_q \binom{\alpha}{\alpha'} \binom{\beta'}{\beta} \overline{N}^{d - 1, \delta'}(\alpha', \beta').
\end{equation}
The second sum runs through all $\alpha', \beta', \delta'$ satisfying 
\begin{align*}
& \alpha'_i \leq \alpha_i, \beta'_j \geq \beta_j,\\
& \sum_i i \alpha'_i + \sum_j j \beta'_j = d - 1,\\
& \delta' = \delta - d + 1 + \sum_j (\beta'_j - \beta_j),
\end{align*}
and for partitions $\alpha, \alpha'$ one sets
\begin{equation*}
\binom{\alpha}{\alpha'} = \prod_i \binom{\alpha_i}{\alpha'_i}. 
\end{equation*}
The initial conditions for \eqref{capharris} are given by $\overline{N}^{d, \delta}(\alpha, \beta) = 0$ whenever 
\begin{equation*}
\binom{d+2}{2} - d -1 + \sum_j \beta_j - \delta \leq 0, 
\end{equation*}
except for the single case
\begin{equation*}
\overline{N}^{1, 0}(e_1, 0) = 1.
\end{equation*}

As discussed at length in \cite{vivek} section 1, the invariants $\overline{N}^{d, \delta}(\alpha, \beta)$ conjecturally give the most natural $q$-deformation of the classical (relative) Severi degrees; in particular it is conjectured (\cite{vivek} Conjecture 75) that for fixed $\delta$ and $d \geq \frac{d}{2} + 1$ the refined Severi degrees $\overline{N}^{d, \delta}(0, d e_1)$ are expressible in terms of the Hodge theory of the relative Hilbert schemes of points over the linear system $|\mathcal{O}_{\PP^2}(d)|$, providing a refinement of the analogous fundamental result in \cite{richard} which links $N^{d, \delta}(0, d e_1)$ to the topological Euler characteristic of the same Hilbert schemes.

In \cite{vivek} section 6 the authors sketch the definition of Block-G\"ottsche tropical invariants $\overline{N}^{d, \delta}_{\rm trop}(\alpha, \beta)$ which enumerate tropical curves immersed in $\R^2$ via the Block-G\"ottsche multiplicity of our Definition \ref{BGmult}, and announce a proof by Block and G\"ottsche that the invariants $\overline{N}^{d, \delta}_{\rm trop}(\alpha, \beta)$ satisfy the $q$-Caporaso-Harris recursion \eqref{capharris}, with the same initial conditions as the $N^{d, \delta}(\alpha, \beta)$. This result, together with a complete treatment of the invariants $\overline{N}^{d, \delta}_{\rm trop}(\alpha, \beta)$, has very recently appeared in \cite{gottsche} (Theorem 1.1). Therefore one has in general
\begin{equation*}
\overline{N}^{d, \delta}(\alpha, \beta) = \overline{N}^{d, \delta}_{\rm trop}(\alpha, \beta).
\end{equation*}
It follows that 
\begin{align*}
\widehat{N}[({\bf P}^d_1, {\bf P}^d_2)] &= \widehat{N}^{\rm trop}[({\bf P}^d_1, {\bf P}^d_2)]\quad \textrm{(by our Definition \ref{refinedGW})}\\
&= \overline{N}^{d, \delta_d}_{\rm trop}(0, e_d)\quad \textrm{(by the definition of \cite{vivek} section 6})\\
&= \overline{N}^{d, \delta_d}(0, e_d)\quad \textrm{(by \cite{gottsche} Theorem 1.1)}. 
\end{align*}
Thus our $q$-deformed invariants for $\PP^2$ given by $\widehat{N}[({\bf P}^d_1, {\bf P}^d_2)]$ are refined Severi degrees of $\PP^2$ in the sense of \cite{vivek} section 5, and can be computed effectively via the $q$-Caporaso-Harris recursion \eqref{capharris}. From this point of view, we may regard the invariants $\widehat{N}[({\bf P}_1, {\bf P}_2)]$ as a generalisation of the genus $0$ refined Severi degrees, allowing more general incidence conditions and weighted projective planes.\\ 

\begin{exa}
We compute $\widehat{N}[(1 + 1 + 1, 1 + 1 + 1)]$ via the $q$-Caporaso-Harris recursion \eqref{capharris}. By the discussion above, this is the ``$q$-number" of nodal cubics in $\PP^2$ which pass through $6$ prescribed points and are maximally tangent to a line. We find
\begin{align*}
\widehat{N}[(1 + 1 + 1, 1 + 1 + 1)] &= \overline{N}^{3, 1}(0, e_3)\\
&= [3]_q \overline{N}^{3, 1}(e_3, 0)\\
&= [3]_q \left([2]_q  \overline{N}^{2, 0}(0, e_2) + \overline{N}^{2, 1}(0, 2 e_1)\right)\\
&= [3]_q \left([2]_q  [2]_q + 2\right).
\end{align*}
This is a deformation of the classical result $\widehat{N}(1 + 1 + 1, 1 + 1 + 1) = \overline{N}^{3, 1}(0, e_3) = 18$ (see e.g. \cite{gps} section 6.4). The same result can be obtained tropically. For example, choosing $6$ points in $\R^2$ given by
\begin{align*}
\left(- k, 0 \right), \left(- k, \frac{1}{2}\right), \left(- k, \frac{3}{2}\right); \left(0, - k\right), \left(1, - k\right), \left(2, - k\right)    
\end{align*}
for sufficiently large $k > 0$ the invariant $\widehat{N}[(1 + 1 + 1, 1 + 1 + 1)]$ enumerates precisely one rational tropical curve of multiplicity $[3]_q [2]_q  [2]_q$, and two curves of multiplicity $[3]_q$, passing through the prescribed points, and with a single outgoing infinite edge.
\end{exa}

\subsection{Nonprimitive classes.} Let $(|{\bf P}_1|, |{\bf P}_2|)$ be primitive. For $k \geq 1$, the $q$-deformed invariant $\widehat{N}[(k {\bf P}_1, k {\bf P}_2)]$ is in general a Laurent polynomial in $\Q[q^{\pm \frac{1}{2}}]$. It seems natural to ask if the BPS structure described in \cite{gps} sections 6.2 and 6.3 admits a $q$-deformation. We do not know at present how to deform the multi-cover formula \eqref{BPS}, but we can at least give some expectation concerning integrality. 

As pointed out in \cite{gps} section 6.3, the actual BPS invariants $n_{d}[({\bf P}_1, {\bf P}_2)]$ (defined through \eqref{BPS}) are integral if and only if the numbers $n'_{d}[({\bf P}_1, {\bf P}_2)]$, uniquely determined by the simpler formula
\begin{equation}\label{BPS2}
N[(k{\bf P}_1, k{\bf P}_2)] = \sum_{d\geq 1, d | k} \frac{(-1)^{d-1}}{d^2} n'_d[({\bf P}_1, {\bf P}_2)]
\end{equation}
are integral. In view of \eqref{qRam}, the relation \eqref{BPS2} has the natural $q$-deformation 
\begin{equation*}
\widehat{N}[(k{\bf P}_1, k{\bf P}_2)] = \sum_{d\geq 1, d | k} \frac{(-1)^{d-1}}{d [d]_q} \widehat{n}'_d[({\bf P}_1, {\bf P}_2)],
\end{equation*}
which uniquely determines Laurent polynomials $\widehat{n}'_d[({\bf P}_1, {\bf P}_2)] \in \Q[q^{\pm \frac{1}{2}}]$. In analogy to \cite{gps} Conjecture 6.2 one may expect that integrality holds, namely $\widehat{n}'_d[({\bf P}_1, {\bf P}_2)] \in \Z[q^{\pm \frac{1}{2}}]$.\\

\begin{exa} To $q$-deform the BPS calculation $n_2[(1, 1+1)] = 0$, we first enumerate tropical embeddings by hand, computing the Block-G\"ottsche invariants 
\begin{align*}
\widehat{N}^{\rm trop}((2), (2, 2)) &= [4]^2_q,\\
\widehat{N}^{\rm trop}((2), (1, 1, 2)) &= [2]^2_q [4]_q,\\
\widehat{N}^{\rm trop}((2), (1, 1, 1, 1)) &= [2]^4_q,\\
\widehat{N}^{\rm trop}((1,1), (2, 2)) &= [2]^2_q [4]_q,\\
\widehat{N}^{\rm trop}((1,1), (1, 1, 2)) &= 2 [2]^2_q + [2]_q [4]_q,\\
\widehat{N}^{\rm trop}((1,1), (1, 1, 1, 1)) &= 2 [2]_q + [2]^3_q.\\
\end{align*}
Summing these up with the correct $q$-ramification factors according to Definition \ref{refinedGW} yields
\begin{equation*}
\widehat{N}[(2, 2 + 2)] = -\frac{1}{2[2]_q}
\end{equation*}
as expected; thus in this case we have
\begin{equation*}
\widehat{n}'_2[(1, 1 + 1)] = \widehat{N}[(2, 2 + 2)] + \frac{1}{2[2]_q} \widehat{N}[(1, 1 + 1)] = 0.
\end{equation*}
\end{exa}

\noindent The lack of effective ways to compute $\widehat{N}[(k{\bf P}_1, k{\bf P}_2)]$ prevents more interesting checks of integrality for now.  

\subsection{Relation to motivic Donaldson-Thomas invariants.} We make a brief remark motivated by the very interesting work in progress of P. Bousseau and R. Thomas \cite{pierrick}.  As we mentioned in section \ref{background}, the invariants $\widehat{N}[({\bf P}_1, {\bf P}_2)]$ are defined in \cite{gps} section 0.4 by enumerating rational curves on the surface $p\!:S \to \PP(|{\bf P}_1|, |{\bf P}_2|, 1)$ which is the blowup of the weighted projective plane at the points where we impose incidence conditions.  The incidence conditions are recorded in the homology class 
\begin{equation*}
\beta = p^* \beta_k - \sum_{i, j} p_{ij} E_{ij},
\end{equation*}
where $k$ is the greatest common divisor of $(|{\bf P}_1|, |{\bf P}_2|)$, $\beta_k$ is the unique class with intersection numbers $\beta_k \cdot D_i = |{\bf P}_i|$, $\beta_k \cdot D_{\rm out} = k$, and the $E_{i j}$ denote the exceptional divisors. In their work in progress, Bousseau and Thomas relate the relative Gromov-Witten invariant $\widehat{N}[({\bf P}_1, {\bf P}_2)]$ on $S$ to certain Joyce-Song invariants virtually enumerating coherent sheaves on the total space of the canonical bundle $\pi\!: K_S \to S$. More precisely, one has
\begin{equation*}
N[({\bf P}_1, {\bf P}_2)] = \overline{\Omega}_{K_S}[({\bf P}_1, {\bf P}_2)] 
\end{equation*}
where $\overline{\Omega}_{K_S}[({\bf P}_1, {\bf P}_2)]$ denotes the Joyce-Song $\Q$-valued count (see \cite{joyce}) virtually enumerating rank $0$ sheaves on $K_S$ with vanishing holomorphic Euler characteristic and support class $\beta$. Therefore it seems natural to expect an equality
\begin{equation*}
\widehat{N}[({\bf P}_1, {\bf P}_2)] = \overline{\Omega}^{\rm ref}_{K_S}[({\bf P}_1, {\bf P}_2)] 
\end{equation*}
with the corresponding motivic Donaldson-Thomas invariant $\overline{\Omega}^{\rm ref}_{K_S}[({\bf P}_1, {\bf P}_2)]$ in the sense of \cite{ks}. As far as we know the construction of motivic Donaldson-Thomas invariants in this generality is still partly conjectural. 
\appendix
\setcounter{secnumdepth}{0}
\section{}

We sketch a proof of the invariance of the counts $\widehat{N}^{\rm trop}_{(\alpha_1, \alpha_2)}({\bf w})$ under a general choice of ends. We follow the direct approach of \cite{gathmark} in the numerical case (although their methods are more general and work for arbitrary genus). As we mentioned this is a very special case of the approach in the proof of \cite{mikhalkin} Theorem 1. According to the first part of the proof of Theorem 4.8 in \cite{gathmark} we only need to check that $\widehat{N}^{\rm trop}_{(\alpha_1, \alpha_2)}({\bf w})$ remains constant when we cross a codimension $1$ locus which corresponds to a rational curve $C$ with a single $4$-valent vertex $V$ (Figure \ref{codim1}). 
\begin{figure}[ht] 
\centerline{\includegraphics[scale=.7]{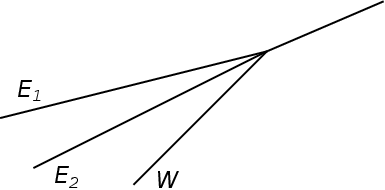}}
\caption{A codimension 1 singularity - first case}
\label{codim1}
\end{figure}
It is enough to show that $\widehat{N}^{\rm trop}_{(\alpha_1, \alpha_2)}({\bf w})$ remains constant when we move just one of the ends $\mathfrak{d}$ with direction $\alpha_2$. Since our curves are rational, the edges $E_1, E_2, W$ have pairwise disjoint sets of ancestors - ends $\mathfrak{d}$ from which they derive. We can assume that the way we degenerate to Figure \ref{codim1} is by moving a single end $\mathfrak{d}$ in the ancestors of $W$: the other two cases (when we displace an ancestor of $E_1, E_2$) are completely analogous. Let us denote by $u_1, u_2$ and $u$ the primitive integral vectors pointing in the direction of $E_1, E_2$ and $W$ respectively, and set $v_1 = w(E_1) u_1$, $v_2 = w(E_2) u_2$ and $w = w(W) u$. Write $\mathfrak{d}^{\pm}$ for the end $\mathfrak{d} \pm \eps \alpha^{\perp}_2$ for sufficiently small $\eps > 0$. In the following we denote by $\mu$ the ordinary slope of vectors in $\R^2$; this should cause no confusion with the tropical multiplicity. Assume at first that 
\begin{equation*}
\mu(v_2 + w) < \mu(v_1).
\end{equation*}
Then there are precisely two curves with $\mathfrak{d}$ replaced by $\mathfrak{d}^-$: nearby Figure \ref{codim1}, they are the curves in Figures \ref{smoothing2}, \ref{smoothing3}.
\begin{figure}[H]
\centerline{\includegraphics[scale=.7]{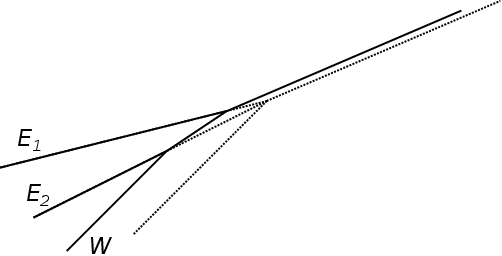}}
\caption{$\mathfrak{d} \to \mathfrak{d}^-$ (first curve)}
\label{smoothing2}
\end{figure}  

\begin{figure}[H]
\centerline{\includegraphics[scale=.7]{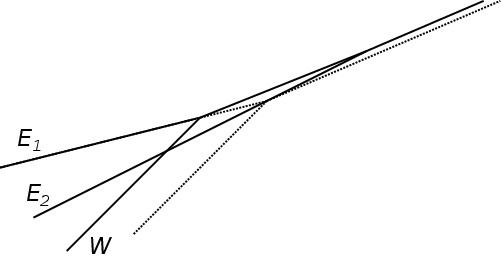}}
\caption{$\mathfrak{d} \to \mathfrak{d}^-$ (second curve)}
\label{smoothing3}
\end{figure}  

\noindent Also there is precisely one curve with $\mathfrak{d}$ replaced by $\mathfrak{d}^+$: nearby Figure \ref{codim1}, it looks as in Figure \ref{smoothing1}.

\begin{figure}[H]
\centerline{\includegraphics[scale=.7]{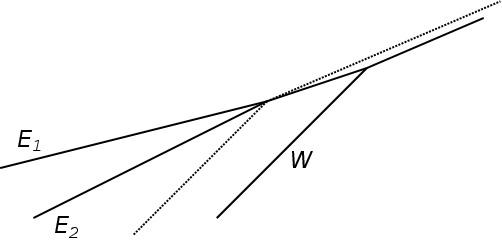}}
\caption{$\mathfrak{d} \to \mathfrak{d}^+$}
\label{smoothing1}
\end{figure}  

\noindent Let us check that $\widehat{N}^{\rm trop}_{(\alpha_1, \alpha_2)}({\bf w})$ remains constant through the codimension 1 wall. We set $m_q = \prod_{V' \neq V}\mu_q(V')$, the product over all trivalent vertices of $C$. The contribution of the $\mathfrak{d}^-$ curves is
\begin{align*}
&m_q (q^{\frac{1}{2}} - q^{-\frac{1}{2}})^{-2}(q^{\frac{1}{2} \bra v_2, w\ket} - q^{-\frac{1}{2} \bra v_2, w\ket})(q^{\frac{1}{2} \bra v_1,  v_2 + w\ket} - q^{- \frac{1}{2} \bra v_1, v_2 + w\ket})+\\
&m_q (q^{\frac{1}{2}} - q^{-\frac{1}{2}})^{-2}(q^{\frac{1}{2} \bra v_1, w\ket} - q^{- \frac{1}{2} \bra v_1, w\ket})(q^{\frac{1}{2} \bra v_1 + w,  v_2\ket} - q^{- \frac{1}{2} \bra v_1 + w, v_2\ket}),
\end{align*}
while the $\mathfrak{d}^+$ curve contributes
\begin{equation*}
m_q (q^{\frac{1}{2}} - q^{-\frac{1}{2}})^{-2}(q^{\frac{1}{2} \bra v_1, v_2\ket} - q^{- \frac{1}{2} \bra v_1, v_2\ket})(q^{\frac{1}{2} \bra v_1 + v_2,  w\ket} - q^{- \frac{1}{2} \bra v_1 + v_2, w\ket}).
\end{equation*}
Expanding this out we find a sum of four terms,
\begin{align*}
a_1 &= m_q (q^{\frac{1}{2}} - q^{-\frac{1}{2}})^{-2} q^{\frac{1}{2} (\bra v_1, v_2\ket + \bra v_1, w\ket + \bra v_2, w\ket )},\\ 
a_2 &= - m_q (q^{\frac{1}{2}} - q^{-\frac{1}{2}})^{-2} q^{\frac{1}{2} (\bra v_1, v_2\ket - \bra v_1, w\ket - \bra v_2, w\ket )},\\
a_3 &= -m_q (q^{\frac{1}{2}} - q^{-\frac{1}{2}})^{-2} q^{\frac{1}{2}  ( -\bra v_1, v_2\ket + \bra v_1, w\ket + \bra v_2, w\ket )},\\
a_4 &= m_q (q^{\frac{1}{2}} - q^{-\frac{1}{2}})^{-2} q^{- \frac{1}{2} ( \bra v_1, v_2\ket + \bra v_1, w\ket + \bra v_2, w\ket )}.
\end{align*}
On the other expanding the contribution of Figure \ref{smoothing2} we find
\begin{equation*}
m_q (q^{\frac{1}{2}} - q^{-\frac{1}{2}})^{-2}(a_1 - q^{\frac{1}{2} ( \bra v_2, w\ket - \bra v_1, v_2\ket - \bra v_1, w\ket )} - q^{\frac{1}{2} ( \bra v_1, v_2\ket + \bra v_1, w\ket - \bra v_2, w\ket )} + a_4),
\end{equation*}
and similarly for Figure \ref{smoothing3}
\begin{equation*}
m_q (q^{\frac{1}{2}} - q^{-\frac{1}{2}})^{-2} (q^{\frac{1}{2} ( \bra v_1, w\ket + \bra v_1, v_2\ket + \bra w, v_2\ket )} + a_2 + a_3 + q^{- \frac{1}{2}  (\bra v_1, w\ket + \bra v_1, v_2\ket +\bra w, v_2\ket )}).
\end{equation*}
So we immediately check that the extra terms cancel out and $\widehat{N}^{\rm trop}_{(\alpha_1, \alpha_2)}({\bf w})$ is preserved. The other situation that we need to consider is when 
\begin{equation*}
\mu(v_2 + w) \geq \mu(v_1).
\end{equation*}
Suppose that the strict inequality holds. This corresponds to a modified picture for the $4$-valent vertex, as in Figure \ref{codim1_B}.
\begin{figure}[H]
\centerline{\includegraphics[scale=.7]{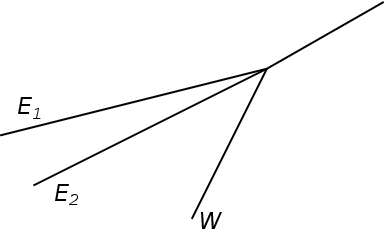}}
\caption{A codimension 1 singularity - second case}
\label{codim1_B}
\end{figure}
\noindent In this case there is a single curve when we replace $\mathfrak{d}$ by $\mathfrak{d}^-$, as in Figure \ref{smoothing2_B}.
\begin{figure}[H]
\centerline{\includegraphics[scale=.7]{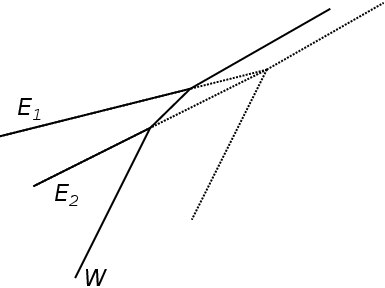}}
\caption{$\mathfrak{d} \to \mathfrak{d}^-$}
\label{smoothing2_B}
\end{figure}  
\noindent And there are two curves with $\mathfrak{d}$ replaced by $\mathfrak{d}^+$, as in Figures \ref{smoothing1_B} and \ref{smoothing3_B}.
\begin{figure}[H]
\centerline{\includegraphics[scale=.7]{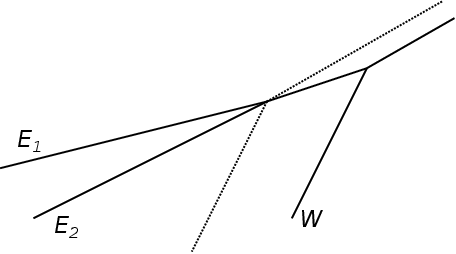}}
\caption{$\mathfrak{d} \to \mathfrak{d}^+$ (first curve)}
\label{smoothing1_B}
\end{figure}  
\begin{figure}[H]
\centerline{\includegraphics[scale=.7]{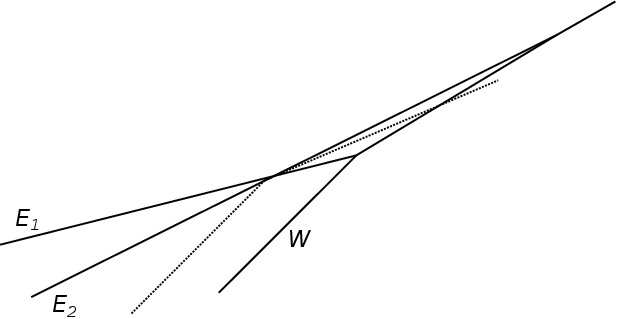}}
\caption{$\mathfrak{d} \to \mathfrak{d}^+$ (second curve)}
\label{smoothing3_B}
\end{figure}  
\noindent The contribution of the $\mathfrak{d}^-$ curve is 
\begin{equation*}
m_q (q^{\frac{1}{2}} - q^{-\frac{1}{2}})^{-2}(q^{\frac{1}{2} \bra v_2, w\ket} - q^{-\bra v_2, w\ket})(q^{\frac{1}{2} \bra v_1, v_2 + w\ket} - q^{- \frac{1}{2}  \bra v_1, v_2 + w\ket}),
\end{equation*}
which we expand as the sum of four terms
\begin{align*}
a'_1 &= m_q (q^{\frac{1}{2}} - q^{-\frac{1}{2}})^{-2} q^{\frac{1}{2} ( \bra v_2, w\ket + \bra v_1, v_2\ket + \bra v_1, w\ket )},\\ 
a'_2 &= - m_q (q^{\frac{1}{2}} - q^{-\frac{1}{2}})^{-2} q^{\frac{1}{2} ( \bra v_2, w\ket - \bra v_1, v_2\ket - \bra v_1, w\ket )},\\
a'_3 &= -m_q (q^{\frac{1}{2}} - q^{-\frac{1}{2}})^{-2} q^{\frac{1}{2}  ( -\bra v_2, w\ket + \bra v_1, v_2\ket + \bra v_1, w\ket )},\\
a'_4 &= m_q (q^{\frac{1}{2}} - q^{-\frac{1}{2}})^{-2} q^{- \frac{1}{2} ( \bra v_2, w\ket + \bra v_1, v_2\ket + \bra v_1, w\ket )}.
\end{align*}
The contribution of Figure \ref{smoothing1_B} is
\begin{align*}
&m_q (q^{\frac{1}{2}} - q^{-\frac{1}{2}})^{-2}(q^{\frac{1}{2} \bra v_1, v_2\ket} - q^{-\frac{1}{2} \bra v_1, v_2\ket})(q^{\frac{1}{2} \bra v_1 + v_2, w\ket} - q^{- \frac{1}{2} \bra v_1 + v_2, w\ket})\\
&= m_q (q^{\frac{1}{2}} - q^{-\frac{1}{2}})^{-2}(a'_1 - q^{\frac{1}{2} ( \bra v_1, v_2\ket - \bra v_1, w\ket  - \bra v_2, w\ket )} - q^{ \frac{1}{2}  (- \bra v_1, v_2\ket + \bra v_1, w\ket  + \bra v_2, w\ket )} + a'_4),
\end{align*}
and that of Figure \ref{smoothing3_B}
\begin{align*}
&m_q (q^{\frac{1}{2}} - q^{-\frac{1}{2}})^{-2}(q^{\frac{1}{2} \bra v_1, w\ket} - q^{- \frac{1}{2} \bra v_1, w\ket})(q^{\frac{1}{2} \bra v_2, v_1 + w\ket} - q^{- \frac{1}{2}  \bra v_2, v_1 + w\ket})\\
&= m_q (q^{\frac{1}{2}} - q^{-\frac{1}{2}})^{-2}(q^{\frac{1}{2} ( \bra v_1, w\ket + \bra v_2, w\ket  + \bra v_2, v_1 \ket )} + a'_2 + a'_3 + q^{- \frac{1}{2} (\bra v_1, w\ket + \bra v_2, v_1\ket + \bra v_2, w\ket )}).
\end{align*}
Again the extra terms cancel out and $\widehat{N}^{\rm trop}_{(\alpha_1, \alpha_2)}({\bf w})$ is preserved.

\bibliographystyle{amsalpha}

\end{document}